\documentclass{article}

\usepackage{amsmath}
\usepackage{amssymb}
\usepackage{amsthm}
\usepackage{hyperref}

\DeclareMathOperator*{\argmin}{arg\,min}

\theoremstyle{plain}
\newtheorem{theorem}{Theorem}
\newtheorem{proposition}[theorem]{Proposition}
\newtheorem{lemma}[theorem]{Lemma}
\newtheorem{corollary}[theorem]{Corollary}

\theoremstyle{definition}

\newtheorem{remark}[theorem]{Remark}

\newtheorem{hyp}[theorem]{Assumption}

\newcommand{\E}{\mathbb{E}}
\newcommand{\R}{\mathbb{R}}
\newcommand{\N}{\mathbb{N}}

\newcommand{\Var}{\mbox{Var}}

\newcommand{\cM}{\mathcal{M}}

\newcommand{\score}{\rho}
\newcommand{\test}{\phi}
\newcommand{\window}{h}

\newcommand{\tangentspace}{T_x \mathcal{M}}

\newcommand{\projection}{\pi}

\renewcommand{\div}{\nabla \cdot}
\newcommand{\ABS}[1]{\left| #1 \right|}

\renewcommand{\det}{\mbox{det}}

\newcommand{\density}{f}

\begin{document}

\title{Optimal score function estimation via derivative constraints}

\author{Thomas Bonis\thanks{LAMA, Univ Gustave Eiffel, UPEM, Univ Paris Est Creteil, CNRS, F-77447, Marne-la-Vallée, France; E-mail: \texttt{thomas.bonis@univ-eiffel.fr}},  Thanh Mai Pham Ngoc\thanks{LAGA, CNRS, UMR 7539, Universit\'e Sorbonne Paris Nord, Villetaneuse, France; E-mail: \texttt{phamngoc@math.univ-paris13.fr}},  Viet Chi Tran\thanks{Univ. Lille, CNRS, Inria, UMR 8524 - Laboratoire Paul Painlevé, F-59000 Lille, France; E-mail: \texttt{viet-chi.tran@inria.fr}}}

\maketitle

\begin{abstract}
We provide a statistical theory explaining why empirical score matching can achieve minimax optimal performance for diffusion generative models when the data distribution is supported on a low-dimensional manifold. More precisely, we establish convergence rates for empirical risk minimization over suitably constrained Sobolev classes and show that, under appropriate regularity conditions, the resulting score estimators lead to minimax optimal Wasserstein-1 rates for the generated distribution, up to logarithmic factors. Our results provide sufficient conditions for score-based generative modeling to generalize from the data instead of overfitting, while the resulting statistical rates depend on the intrinsic dimension of the data manifold rather than its ambient dimension.
\end{abstract}

\section{Introduction}

Score-based generative models (SGMs) are a family of algorithms aiming at sampling new data points from an unknown probability measure $\mu$ on $\R^D$, possibly supported on a $d$-dimensional submanifold $\mathcal{M}$, using only an independent and identically distributed (i.i.d.) sample $X_1, \dots, X_n \sim \mu$, see \cite{Song2020}. In this paper, we will always denote by $d$ the dimension of interest and by $D$ the dimension of the ambient space when relevant.
The idea behind SGM is to use a diffusion process $(\overrightarrow{X}_t)_{t \geq 0}$ with initial distribution $X_0 \sim \mu$ and whose invariant measure $\gamma$ is the centered multivariate Gaussian measure on $\R^D$ with the identity matrix $I_D$ as covariance matrix. This diffusion can be chosen to be the Ornstein-Uhlenbeck process 
\begin{equation}\label{eq:OU}
d\overrightarrow{X}_t=-\overrightarrow{X}_t\ dt+\sqrt{2}\ dB_t,\qquad \overrightarrow{X}_0 = X \sim \mu,
\end{equation}
where $B$ is a standard $D$-dimensional Brownian motion independent of $X_0$. There exists a unique strong solution to Equation~\eqref{eq:OU} associated with any $X_0$ and $(B_t)_{t\geq 0}$, and given for $t > 0$ by:
\begin{equation}
\label{eq:OUconvo}
    \overrightarrow{X}_t= e^{-t} X + \sqrt{1-e^{-2t}} Z_t,
\end{equation}
where $Z_t=\int_0^t \sqrt{2}e^{-(t-s)}dB_s / \sqrt{1-e^{-2t}}$ is a centered Gaussian random variable with covariance matrix $I_D$. 


For each $t >0$, and regardless of the initial condition $\mu$, the distribution of $\overrightarrow{X}_t$ admits a density $f^{\mu_t}$ with respect to the Lebesgue measure on $\R^D$:
\begin{align}
f^{\mu_t}(x)= & \big(2\pi (1-e^{-2t})\big)^{-D/2} \int_{\cM}    e^{-\frac{\|x-x_0 e^{-t}\|^2}{2(1-e^{-2t})}}  \mu(dx_0).\label{eq:densite_h-1}
\end{align}
This allows us to define the reverse process $(\overleftarrow{X}_t)_{t \in [0,T]}$, whose initial distribution is the distribution of $\overrightarrow{X}_T$, and such that for any $t \in [0,T]$, $\overrightarrow{X}_t$ and $\overleftarrow{X}_{T-t}$ follow the same distribution, see
\cite{chungwalsh,dellacheriemeyer,haussmannpardoux,nagasawa}. More precisely, for all $t\in [0,T)$, 
\[
d\overleftarrow{X}_t = \big(\overleftarrow{X}_t + 2 \score^\star_{T-t}(\overleftarrow{X}_t)\big) dt + \sqrt{2} dW_t, 
\]
where $W$ is a standard $D$-dimensional Brownian motion and where $\score^\star_t$ is the score function defined for any $t>0$ as
\begin{equation}\label{def:score-diff}
    \score^\star_t = \nabla \log f^{\mu_t}=\frac{\nabla f^{\mu_t}}{f^{\mu_t}}.
\end{equation}
Therefore, running the backward process $(\overleftarrow{X}_t)_{t \in [0,T]}$ up to time $T$ and starting from the initial distribution $\gamma$ should generate a sample close to $\mu$. Since $\mu$ is unknown, one does not have access to the score function $\score^\star_{T-t}$. The latter should be estimated from the data sample in order to run an approximate backward process using this estimate. Running this backward process then provides a sample from a measure $\hat{\mu}^{SGM}_n$. In order to evaluate the performance of this algorithm, one can then compare $\hat{\mu}^{SGM}_n$ to the target measure $\mu$. One can then  distinguish between three types of errors, see e.g. \cite{Bortoli}: 
\begin{itemize}
    \item An initialization error: the actual measure of $\overrightarrow{X}_T$ is unknown and is approximated by a Gaussian measure $\gamma$. The larger $T$ is, the smaller this error becomes. 
    \item A discretization error: the backward process $\overleftarrow{X}_{T-t}$ cannot be computed in continuous time and one must instead rely on a discrete approximation whose accuracy depends on a user-selected step-size.
    \item A statistical error: the score functions $(\score_t^\star)_{t > 0}$ are unknown and need to be estimated from the sample $X_1, \dots, X_n$.
\end{itemize}
In this paper, we are interested in the third error in order to understand the empirical success of these algorithms and possibly improve them. More precisely, we provide sufficient conditions for score function estimators usually employed in SGM to be ``optimal" in a statistical sense. \medbreak

We start in Section \ref{section:penalisation} with the problem of estimating the score function of a probability measure $\mu$ with density $f^\mu$ from an i.i.d. sample of $\mu$. For simplicity and because we are interested in the score function associated with $\mu$, we consider the simpler case where $\mu$ is a measure with density $f^\mu$ bounded away from $0$ over the $d$-dimensional flat torus $\mathcal{T}^d$. This problem is thus equivalent to estimating both the density and its gradient. From the works of \cite{Stone-82, Comte-Sacko-Duval}, it is known that the minimax rate for estimating the $\nu$-th derivative of a density of regularity $s$ supported on $\mathbb{R}^d$ is of order $n^{-(s-\nu)/(2s+d)}$ for the $L^2$-norm. Therefore, in our case, we expect that if $f^\mu$ is $s$-times differentiable, then the score function can be estimated at a rate of order $n^{-(s-1)/(2s+d)}$. However, these optimal rates are usually achieved using non-parametric kernel density estimators, which is quite different from the empirical risk minimization approach used for SGM. Using a similar kernel smoothing approach, \cite{Wibisono2024} obtains the minimax rate for $s=2$, up to some logarithmic terms, for sub-Gaussian measures in $\R^d$. In this work, we show that such minimax rates can be achieved for all $s \geq 2$ through an empirical risk minimization similar to the one used for SGM provided the hypothesis class is constrained to an adequate Sobolev ball and a penalization on the ($s-1$)-th derivatives is added to the empirical loss. Using chaining arguments, we show that such conditions guarantee that the variance of the estimator is controlled while allowing for a small enough bias. In a way, our work can be viewed as a continuation of \cite{Silverman1982} which considered the problem of density estimation by penalized log-likelihood minimization which corresponds to estimating a density function by minimizing an empirical Kullback-Leibler divergence. In contrast, if the hypothesis class is constrained to only contain actual score functions associated with a density, then score function estimation is equivalent to density estimation through minimization of an empirical relative Fisher information. \medbreak

We then return to the problem of estimating the score function in the diffusion setting used in SGM in Section \ref{sec:avecdiffusion}. In this case, we want to estimate the various score functions $(\score^\star_t)_{t > 0}$ of the measure $\mu_t$ from the i.i.d. sample $X_1, \dots, X_n \sim \mu$. Remark that, by Equation~\eqref{eq:OUconvo}, $\mu_t$ admits a $C^\infty$ density $f^{\mu_t}$ on $\R^D$ for all $t>0$. From the results of Section \ref{section:penalisation}, one could then expect to achieve the rate $n^{-(s-1)/(2s+D)}$ for any $s \in \mathbb{N}$ and thus be as close as desired to the parametric rate $n^{-1/2}$. However, remark that such a rate would involve the ambient dimension $D$ and not the intrinsic one $d$. Furthermore, when studying the performance of SGM, we require estimating score functions at all times $t > 0$ and, while $f^{\mu_t}$ is indeed in $C^\infty$ for all times $t > 0$, its derivatives usually explode when $t \rightarrow 0$. However, one can still assess the performance of a collection of score function estimators $(\hat{\score}_t)_{t > 0}$ with respect to the sampling problem by using it in an SGM and comparing the sampled measure $\hat{\mu}^{SGM}_n$ with $\mu$. One can then consider a distance on the space of measures, such as the Wasserstein-$1$ distance, between $\hat{\mu}^{SGM}_n$ and $\mu$ in order to evaluate the performance of the score function estimators through the lens of measure estimation. Fortunately, minimax rates for measure estimation in Wasserstein-$1$ distance are known: if $\mu$ is a probability measure with $s$-smooth density on the hypercube $[-1,1]^d$ for $d \geq 3$ then it can be estimated with a minimax rate $n^{-(s+1)/(2s + d)}$ \cite{NilesWeed2022}. Under similar assumptions on the measure $\mu$, \cite{oko} claims that the SGM approach can reach optimal rates in this measure estimation sense. Their result is derived by considering a particular structure of neural network which describes score functions of diffusions from spline density estimators. However, such neural networks are completely different from the neural networks used in practice. A more thorough proof was provided by \cite{StephanovitchAaamariLevrard} for more general sub-Gaussian measures of $\R^d$. Roughly, they obtain minimax rates by showing that, if the derivatives of order $s$ of $\log(f^\mu)$ are bounded, then the derivatives of order $s$ of the score function at time $t$ are of order $1/\sqrt{t}$. The score functions can then be estimated by empirical risk minimization over a neural network structure which approximates a Sobolev ball. However, this approach does not hold whenever $\mu$ is not a probability measure with smooth density in $\R^D$ but instead a measure supported on a manifold. This assumption is important to many fields such as computer vision, where images are usually assumed not to lie on the full ambient space $\R^D$ (where $D$ would be the number of pixels of an image) but instead live on a $d$-dimensional manifold where $d \ll D$. In this case, the minimax rate for measure estimation in Wasserstein distance is the same as for the smooth density case \cite{Divol2022} but instead depends on the smoothness $s$ of the density of $\mu$ with respect to the volume measure of the manifold. However, the behavior of the score functions $(\score^\star_t)_{t > 0}$ is very different as these functions (and their derivatives) explode when $t$ gets closer to $0$ \cite{Stanczuk2024}.  Extensions of the approach of \cite{oko} for score function estimation were proposed in \cite{Tang2024, azangulov, zhang2026, fu2026} but also involve ad hoc neural networks. Recent works \cite{Gabriel2025, lyu2025} also studied kernel smoothing but such an approach is also quite different from the empirical risk minimization approach used in practice. Furthermore, the rate obtained by \cite{lyu2025} under the $s=2$ assumption is only of order $n^{-\frac{2}{d+4}}$ which is suboptimal as we would expect $n^{-\frac{3}{d+4}}$ under this assumption. \medbreak

In this work, we derive convergence rates for an empirical risk minimizer $\hat{\score}_t$ as an estimator of $\score^\star_t$ under the assumption that $\mu$ admits an $s$-smooth density $f^\mu$ with respect to the volume measure of a compact manifold $\mathcal{M}$, see Theorem~\ref{thm:maindiffusion}. Then, we show in Corollary~\ref{cor:finalresult} that our estimator can be used in an SGM to achieve the optimal measure estimation rate in Wasserstein-$1$ distance. More precisely, we show that constraining the hypothesis class of vector fields considered in the minimization problem to belong to a Sobolev ball with radius proportional to $1/t$ is sufficient to control the variance while keeping the bias sufficiently small. The variance of our estimator can then be controlled by a chaining argument exploiting the smoothing properties of the Ornstein-Uhlenbeck semigroup. In the spirit of \cite{StephanovitchAaamariLevrard}, we control the bias by studying the regularity of the score functions, showing that they are well approximated by the  considered Sobolev ball. As the variance result does not assume any particular kind of structure regarding the hypothesis class considered, it can also be applied to neural network models used in practice showing that, as long as their derivatives with respect to the input are properly bounded, they will not overfit and will ``generalize" from the data. Of course, in practice, it would be difficult to guarantee such a condition but one can imagine adding a penalization term to the empirical risk minimization problem instead to force the minimization procedure to select smooth enough score functions. In fact, usual penalizations used when training neural networks have been shown to correspond to a weighted penalization of the second derivative for some simple neural network structures \cite{savarese2019look}. If such a result holds for more general neural networks, this could be a reason why they properly estimate score functions despite their very large number of parameters. \medbreak

After recalling some notation in Section \ref{sec:notation}, we study score estimation on the torus in Section \ref{section:penalisation} for an i.i.d. sample without diffusion. This allows us to show the main ideas involved when using empirical risk minimization for estimating the score function. In Section \ref{sec:avecdiffusion}, we return to the question of score estimation in the SGM setting. The proofs are more technical in this case and are postponed to Section \ref{sec:proofs}.

\section{Notation}\label{sec:notation}


For $\nu$ a measure on a manifold $\cM$ (possibly $\R^D$) and $p\geq 1$, we will denote by $L^p(\nu)$ the space of real-valued measurable functions $g$ on $\cM$ such that 
\[
\|g\|^p_{L^p(\nu)} :=\int_{\cM}\ABS{g}^p d\nu<\infty.
\]
When $p = 2$, we also define the associated scalar product between two functions $f$ and $g$ on $\cM$:
\[
\left<f,g\right>_{L^2(\nu)} = \int_{\mathcal{M}} fg d\nu.
\]
When $\nu$ is the volume measure of a manifold $\mathcal{M}$ or the Lebesgue measure of $\R^D$, we also use the notation $L^p(E) = L^p(\nu)$, with $E=\mathcal{M}$ or $\R^D$. \medbreak

For any $k, l \in \N$ and any tensors $x=(x_{i_1,\dots i_k})\in \R^{D^k}$, $y=(y_{i_1,\dots i_l})\in \R^{D^l}$, we define the tensor product between $x$ and $y$, denoted $x \otimes y \in  \R^{D^{k+l}}$ by
\[
\forall i=(i_1,\dots i_{k+l}) \in \{1, \dots D\}^{k+l}, (x \otimes y)_i = x_{i_1, \dots, i_k} y_{i_{k+1}, \dots, i_{k+l}} .
\]
For any $x \in \R^D$ and $k \in \N$, let $x^{\otimes k} \in \R^{D^k}$ such that 
\[
\forall i =(i_1,\dots i_k) \in \{1, \dots D\}^k,\quad  x^{\otimes k}_i = x_{i_1}\dots x_{i_k}. 
\]
We also denote the scalar product between two tensors $x,y \in (\R^D)^{\otimes k}$ by
\[
\left<x,y \right> = \sum_{i \in \{1, \dots D\}^{k}} x_i y_i
\]
and the associated Hilbert-Schmidt norm $\|x\|^2 = \left<x,x\right>$.

For any $k \in \mathbb{N}$ and any $k$-times differentiable function $g : \R^D \rightarrow \R$, we denote by $\nabla^k g$ the $k$-tensor function such that 
\[
\forall i \in \{1, \dots ,D\}^k, \quad (\nabla^k g)_i = \frac{\partial^k g}{\partial x_{i_1} \partial x_{i_2} \dots \partial x_{i_k}}
\]
and we extend this definition to any vector field $g:\R^D \rightarrow \R^D$ by
\[
\forall i \in \{1, \dots ,D\}^{k+1}, \quad (\nabla^k g)_i = \frac{\partial^k g_{i_1}}{\partial x_{i_2} \partial x_{i_3} \dots \partial x_{i_{k+1}}}.
\]
For any $k$-tensor $h$, we define the $2$ and $\infty$ norms as we would do for traditional vectors:
\[
\|h\| = \sqrt{\sum_{i \in \{1, \dots ,D\}^{k}} h_i^2}
\]
and
\[
\|h\|_\infty = \max_{i \in \{1, \dots ,D\}^{k}} |h_i|.
\]
For $s \in \N,p \in [1, \infty]$ and any measure $\nu$ on $\R^D$, we consider the Sobolev space of functions $W^{s,p}(\nu)$ defined by 
\[
W^{s,p}(\nu) = \{ g : \R^D \rightarrow \R \mid \forall k \in \{0, 1, \dots, s\}, \|\nabla^k g\| \in L^p(\nu)\}.
\]
For any $g \in W^{s,p}(\nu)$, we define the associated norm
\[
\|g\|_{W^{s,p}(\nu)} = \left( \sum_{k=0}^s \|\nabla^k g\|_{L^p(\nu)}^p \right)^{1/p}.
\]
As for $L^p$ spaces, when $\nu$ is the volume measure of a manifold $\mathcal{M}$ or the Lebesgue measure of $\R^D$, we write $W^{s,p}(E) = W^{s,p}(\nu)$ for $E=\mathcal{M}$ or $\R^D$. These definitions extend to $p=\infty$ with 
\[
\|g\|_{W^{s,\infty}(\nu)} = \sum_{k=0}^s \sup_{x \in E} \|\nabla^k g(x)\|.
\]
We use similar notation and definitions for Sobolev spaces of vector fields. Finally, for any differentiable vector field $g : \R^D \rightarrow \R^D$, we define its divergence by 
\[
\div g = \sum_{i=1}^D \frac{\partial g_i}{\partial x_i}.
\]
\section{Score estimation without diffusion}\label{section:penalisation}
\subsection{Setting}

We study in this section the estimation of the score function $\rho^*$ associated with a probability measure $\mu$ on the flat torus $\mathcal{T}^d$ in order to avoid boundary effects. We will do so by considering a penalized or constrained optimization problem based on the empirical measure $\mu_n$ of the observations $X_1,\dots X_n$ that are i.i.d. random variables drawn from the distribution $\mu$.

\begin{hyp}\label{hyp:torus} The probability measure $\mu$ is supported on the flat $d$-dimensional torus $\mathcal{T}^d$, with a density $f^\mu$ with respect to the volume measure of $\mathcal{T}^d$. This density is supposed to belong to $ W^{s,\infty}(\mathcal{T}^d)$ with $s\geq 2$, and to be lower bounded by a positive constant $\density_{\min}>0$. 
\end{hyp}

\begin{remark}
Since the derivatives of the density are bounded and the density itself is bounded from below, there exists $R > 0$ such that
\begin{equation}\label{eq:bouleR}
    \|\score^\star\|_{W^{s-1, \infty}(\mathcal{T}^d)} = \|\nabla \log f^\mu\|_{W^{s-1, \infty}(\mathcal{T}^d)} \leq R.
\end{equation}
\end{remark}

\subsection{Score function estimator}
\label{sec:nondiffestimator}

In \cite[Theorem 1]{hyvarinien05}, it was shown that minimizing $\|\score^\star - g\|_{L^2(\mu)}$ for a vector field $g \in W^{1,2}(\mu)$ is equivalent to minimizing the following loss
\begin{equation}\label{def:lossL}
L(g) = \int_{\mathcal{T}^d} l_g \, d\mu = \E_{X \sim \mu} [l_g(X)].
\end{equation}
where
\[
\forall x \in \mathcal{T}^d, l_g(x) =\|g(x)\|^2 +  2 \div  g(x).
\]
Indeed, this is a consequence of: 
\begin{lemma}
    \begin{equation}\label{eq:diffL}
L(g)-L(\rho^\star) =\int_{\mathcal{T}^d} \|g-\rho^\star\|^2 d\mu\geq 0.
\end{equation}
\end{lemma}
\begin{proof}To derive this identity, let us notice that for $g\in W^{1,2}(\mu)$, we have 
\[
\int_{\mathcal{T}^d} \|g - \score^\star\|^2 d\mu = \int_{\mathcal{T}^d} \left(- 2  \left<g,\score^\star\right> + \|g\|^2 + \|\score^\star\|^2 \right) d\mu. 
\]
By integration by parts on $\mathcal{T}^d$, we have 
\begin{align*}
\int_{\mathcal{T}^d} - 2  \left<g,\score^\star\right> d\mu & = \int - 2  \left<g, \nabla f^\mu\right> \, dx 
 = \int 2  (\div g )\, f^\mu \, dx 
 =  \int 2  (\div g) \, d\mu.
\end{align*}
Performing a similar integration by parts on $\int \|\score^\star\|^2 d\mu = \int \left<\score^\star, \score^\star\right> d\mu$, we obtain 
\[
\int \|g - \score^\star\|^2 d\mu = \int \left(2  \div (g-\score^\star)  + \|g\|^2 - \|\score^\star\|^2 \right) d\mu,
\]which implies \eqref{eq:diffL}.
\end{proof}

Provided $\score^\star \in W^{1,2}(\mu)$, we can thus define it as the minimizer of this loss as
\begin{equation}
\score^\star = \argmin_{g \in W^{1,2}(\mu)} \|\score^\star - g \|_{L^2(\mu)} = \argmin_{g \in W^{1,2}(\mu)} L(g) . \label{eq:score-argmin}
\end{equation}
Let us now introduce the empirical version of the expression of $L(g)$. For $X_1, \dots, X_n$ i.i.d. random variables drawn from $\mu$ and any vector field $g \in W^{1, \infty}(\mathcal{T}^d)$, let  
\[
\widehat{L}(g) = \frac{1}{n} \sum_{i=1}^n  l_g(X_i).
\]
Then, given a hypothesis class of vector fields $\mathcal{F} \subset W^{1, \infty}(\mathcal{T}^d)$, a reasonable way to estimate $\score^\star$ would be to find the minimizer of $\widehat{L}$ over $\mathcal{F}$. However, in order to obtain optimal rates, a penalization must be added to our loss. For $\lambda > 0$, we define 
\[
L^{\lambda}(g) = \int_{\mathcal{T}^d} l_g \, d\mu + \lambda \|\nabla^{s-1} g\|^2_{L^2(\mathcal{T}^d)}
\]
and its empirical equivalent 
\[
\widehat{L}^\lambda(g) = \frac{1}{n} \sum_{i=1}^n  l_g(X_i) + \lambda \|\nabla^{s-1} g\|^2_{L^2(\mathcal{T}^d)}.
\]
The penalization plays a smoothing role in the sense that it entices us to choose estimators with small derivatives of order $s-1$. It is important to note that the norm of these derivatives is chosen with respect to the volume measure of the torus, as the measure $\mu$ is unknown, hence the $L^2(\mathcal{T}^d)$-norm. Assuming minimizers of this penalized loss exist, we can then define
\begin{equation}\label{eq:score-Llambda}
    \score^\mathcal{F} \in \argmin_{g\in \mathcal{F}} L^\lambda(g),
\end{equation}
and our estimator by
\begin{equation}\label{eq:rhohat}
    \hat{\score} \in \argmin_{g\in \mathcal{F}} \widehat{L}^\lambda(g).
\end{equation}
We finally have to discuss the hypothesis class of functions $\mathcal{F}$ considered. In practice, this class may be a parametrized class of neural networks or a reproducing kernel Hilbert space or something else depending on the application. We expect $\hat{\score}$ to be close to $\score^\star$ but if $\mathcal{F}$ is too large then $\hat{\score}$ may have too large a variance meaning that $\hat{\score}$ will be far away from $\score^\star$. On the other hand, if $\mathcal{F}$ is too small then $\score^\mathcal{F}$ itself will be far away from $\score^\star$. We now introduce a class of functions of a ``good" size. Let $h>0, \ell \in \mathbb{N}^\star$ and consider the constant $R > 0$ defined in Equation~\eqref{eq:bouleR}. Letting $C(s, \ell)$ be a constant depending only on $s$ and $\ell$, we define $\mathcal{F}_{s, R, \window, \ell}$ to be the following class of vector fields: 
\begin{multline}\label{def:classeF}
\mathcal{F}_{s, R, \window, \ell}=\Big\{g\in W^{s+\ell-1, \infty}(\mathcal{T}^d)\ \big|\  \|g\|_{W^{1,\infty}(\mathcal{T}^d)} \leq C(s, \ell) R,\ \\
\forall k \in \{0,\dots,\ell\}, \|g\|_{W^{s-1+k, \infty}(\mathcal{T}^d)} \leq \frac{C(s, \ell) R}{h^{k}}\Big\}.
\end{multline}
Note that the dependence of $\mathcal{F}_{s, R, \window, \ell}$ on $h$ is not so standard in mathematical statistics. Here, it indicates that we are looking for a minimizer with the same regularity as a kernel estimator with bandwidth $h$.

\begin{theorem}\label{rate-score-tore-plat}
Under Assumption \ref{hyp:torus}, if $\mathcal{F} = \mathcal{F}_{s, R, \window, \ell}$ with $\ell$ such that $2(s-2+\ell) > d$ and $\lambda < 1$, then there exists $C > 0$ independent of $n, \window$ and $\lambda$ such that the score estimator $\hat{\score}$ defined in Equation~(\ref{eq:rhohat}) verifies
\[
\E\big[\|\hat{\score} - \score^\star\|_{L^2(\mu)}^2\big] \leq  C \left( \lambda^{\frac{d-2(s+\ell-2)}{(s-1)(d+2(s+\ell-2))}}  h^{-\frac{2d\ell}{d+2(s+\ell-2)}} n^{-\frac{2(s+\ell-2)}{d+2(s+\ell-2)}} + h^{2(s-1)} + \lambda \right).
\]
In particular, for $\lambda = h^{2(s-1)}$, 
\[
\E\big[\|\hat{\score} - \score^\star\|_{L^2(\mu)}^2\big] \leq C (h^{-2 +\frac{2d(2-\ell)}{d+2(s+\ell-2)}} n^{-\frac{2(s+\ell-2)}{d+2(s+\ell-2)}}+ h^{2(s-1)})
\]
and, if in addition $h = n^{-\frac{1}{d+2s}}$,
\[
\E[\|\hat{\score} - \score^\star\|_{L^2(\mu)}^2] \leq 2C n^{ \frac{-2(s-1)}{2s+d}}.
\]
\end{theorem}

The convergence rate established in Theorem \ref{rate-score-tore-plat} can be related to minimax rates known in the nonparametric literature for estimating derivatives of a probability density with smoothness $s$. Since the seminal works of \cite{Stone-82}, and more recently in adaptive frameworks such as \cite{Comte-Sacko-Duval}, it is known that the minimax rate for estimating the $\nu$-th derivative of a density of regularity $s$ supported on $\mathbb{R}^d$ is of order $n^{-2(s-\nu)/(2s+d)}$. Recalling that the score is the derivative of the log-density - and thus a problem at least as difficult as estimating the derivative of the density - it is therefore natural to conjecture that the rate obtained here for score estimation is minimax optimal.

\begin{proof}
We have:
\[
\E\big[\|\hat{\score} - \score^\star\|_{L^2(\mu)}^2\big] \leq 2  \| \score^\mathcal{F}-\score^\star\|_{L^2(\mu)}^2 + 2 \E\big[\|\hat{\score} - \score^\mathcal{F}\|_{L^2(\mu)}^2\big].\]
The two terms on the right-hand side are respectively the bias and variance and are treated in Proposition \ref{prop:biais} and Proposition \ref{prop:variance} in the coming sections.
\end{proof}

\subsection{Bias}

\begin{proposition}\label{prop:biais}
Under Assumption \ref{hyp:torus}, for any $0 < h, \lambda < 1$ and $\ell \in \mathbb{N}$, if $\mathcal{F} = \mathcal{F}_{s, R, h, \ell}$ then there exists $C(s, \ell,R) >0$ such that
\[
\|\score^\mathcal{F} - \score^\star\|_{L^2(\mu)}^2 \leq C(s, \ell,R) (h^{2(s-1)} + \lambda).
\]
\end{proposition}

\begin{proof}For $h>0$, let us consider the kernel approximation $\score_h$ of $\score^\star$ defined as follows. Let $K: \R^+ \rightarrow \R$ be a kernel function supported on $[0,1]$ such that
\begin{itemize}
\item $\int_{\R^d} K(\|x\|) \, dx = 1$;
\item $\forall 1 \leq j \leq s-1, \int_{\R^d} K(\|x\|)x^{\otimes j} \, dx = 0$;
\item $\max_{k \in \{0, \dots, \ell\}} \int_{\R^d} \|\nabla^k K(\|x\|)\| \, dx \leq C(s, \ell, R)$. 
\end{itemize}
For $x \in \mathcal{T}^d$, we define 
\[
\score_h(x) = \int_{\mathcal{T}^d} h^{-d} K\left(\frac{\|y-x\|}{h}\right) \score^\star(y) dy,
\]
where $\|y-x\|$ denotes the geodesic distance between $x$ and $y$ on the torus. By definition of $\score^\mathcal{F}$ as minimizer of $L^\lambda$ in \eqref{eq:score-Llambda} over the class $\mathcal{F}_{s, R, h, \ell}$ defined in \eqref{def:classeF}, if one can show that $\score_h$ belongs to the same class then one would have 
\[
 L^{\lambda}(\score_h) - L(\score^\star) \geq L^{\lambda}(\score^\mathcal{F}) - L(\score^\star) \geq L(\score^\mathcal{F}) - L(\score^\star) = \|\score^\mathcal{F} - \score^\star\|_{L^2(\mu)}^2.
\]
To prove the result of the proposition, it is therefore sufficient to show that 
\begin{equation}\label{etape5}
    L^{\lambda}(\score_h) - L(\score^\star) = \|\score_h - \score^\star\|_{L^2(\mu)}^2 + \lambda \|\nabla^{s-1} \score_h\|^2_{L^2(\mathcal{T}^d)} \leq C(s, \ell, R)(h^{2(s-1)}+\lambda)
\end{equation}
for some $C > 0$.

We now show that $\score_h$ belongs to the class $\mathcal{F}_{s, R, h, \ell}$. Since $\|\score^\star\|_{W^{s-1,\infty}(\mathcal{T}^d)} \leq R$, performing successive integrations by parts yields, for any $k \leq s-1$,
\begin{align*}
\|\nabla^{k} \score_h\|_{L^\infty(\mathcal{T}^d)} & = \left\|\int_{\mathcal{T}^d} h^{-d} \nabla^k K\left(\frac{\|y-x\|}{h}\right) \score^\star(y) dy\right\|_{L^\infty(\mathcal{T}^d)} \\
& = \left\|\int_{\mathcal{T}^d} h^{-d} K\left(\frac{\|y-x\|}{h}\right) \nabla^k \score^\star(y) dy\right\|_{L^\infty(\mathcal{T}^d)} \\
& \leq\|\nabla^k \score^\star(y)\|_{L^\infty(\mathcal{T}^d)} \int_{\mathcal{T}^d} h^{-d} \left\|K\left(\frac{\|y-x\|}{h}\right)\right\| dy \\ 
& \leq  C(s, \ell,R). 
\end{align*}In particular, for $k=s-1$, 
\begin{equation}
    \|\nabla^{s-1} \score_h\|_{L^\infty(\mathcal{T}^d)} \leq  C(s, \ell,R) R. \label{etape4}
\end{equation}
Moreover, for $k \in \{0, \dots, \ell\}$, there exists $C > 0$ such that 
\begin{align*}
\|\nabla^{s-1+k} \score_h\|_{L^\infty(\mathcal{T}^d)} & = \left\|\int_{\mathcal{T}^d} h^{-d} \nabla^{k} K\left(\frac{\|y-x\|}{h}\right) \nabla^{s-1}  \score^\star(y) dy\right\|_{L^\infty(\mathcal{T}^d)} \\
& \leq \frac{ C(s, \ell, R)}{h^{k}}  \| \nabla^{s-1} \score^\star(y)\|_{L^\infty(\mathcal{T}^d)}.
\end{align*}
Thus, $\score_h$ belongs to $\mathcal{F}_{s, R, \window, \ell}$ defined in \eqref{def:classeF}. \medbreak

Finally, using a Taylor expansion, there exists $C > 0$ such that 
\begin{align*}
\|\score_h(x) - \score^\star(x)\| & \leq   \int_{\mathcal{T}^d} h^{-d} \left\|K\left(\frac{\|y-x\|}{h}\right)\right\| \|y-x\|^{s-1} \|\nabla^{s-1} \score^\star\|_{L^\infty(\mathcal{T}^d)} dy \\
& \leq  C(s, \ell,R) h^{s-1}. 
\end{align*}
Combining this with \eqref{etape4} yields \eqref{etape5} and finishes the proof.
\end{proof}

\subsection{Variance}
\label{sec:variancesansdiffusion}
\begin{proposition}\label{prop:variance}
Under Assumption \ref{hyp:torus}, suppose the candidate class $\mathcal{F}$ is convex and is contained in $\mathcal{F}_{s, R, \window, \ell}$ for some $h>0, \lambda < 1$ and $\ell$ such that $2(s+\ell - 2) > d$. Then, there exists $C >0$ independent of $n,\window$ and $\lambda$ such that
\[
\E[\|\hat{\score} - \score^\mathcal{F}\|_{L^2(\mu)}^2] \leq C \lambda^{\frac{d-2(s+\ell-2)}{(s-1)(d+2(s+\ell-2))}}  h^{-\frac{2d\ell}{d+2(s+\ell-2)}} n^{-\frac{2(s+\ell-2)}{d+2(s+\ell-2)}}.
\]
\end{proposition}
\begin{proof}
In the following, we denote by $C$ a generic positive constant independent of $n,\window$ and $\lambda$.
We decompose the proof into several steps.\medbreak

\noindent \textbf{Step 1:} we will first upper-bound $\|\hat{\score} - \score^\mathcal{F}\|_{L^2(\mu)}^2$ and show that
\begin{equation}
  \|\hat{\score} - \score^\mathcal{F}\|_{L^2(\mu)}^2  + \lambda \|\nabla^{s-1} (\hat{\score} - \score^\mathcal{F})\|_{L^2(\mathcal{T}^d)}^2 \leq   \frac{1}{n} \sup_{g \in \mathcal{F}} \sum_{i=1}^n  \tilde{l}_{\score^\mathcal{F}} (X_i) - \tilde{l}_g(X_i),\label{but1}
\end{equation}
where for any function $g\in  \mathcal{F}$, 
\[
\tilde{l}_g = l_g - \E[l_g(X_1)].
\]

We have 
\begin{multline*}
L^{\lambda}(\hat{\score}) - L^{\lambda}(\score^\mathcal{F}) \\
\begin{aligned}
    & = \|\hat{\score} - \score^\star\|_{L^2(\mu)}^2 - \|\score^\mathcal{F} - \score^\star\|_{L^2(\mu)}^2 + \lambda \left( \|\nabla^{s-1} \hat{\score}\|_{L^2(\mathcal{T}^d)}^2 - \|\nabla^{s-1} \score^\mathcal{F}\|_{L^2(\mathcal{T}^d)}^2 \right)\\
& = \left<\hat{\score} - \score^\mathcal{F}, \hat{\score} + \score^\mathcal{F} - 2 \score^\star \right>_{L^2(\mu)} + \lambda\left< \nabla^{s-1}(\hat{\score} - \score^\mathcal{F}), \nabla^{s-1}(\hat{\score} + \score^\mathcal{F}) \right>_{L^2(\mathcal{T}^d)}\\
& = \|\hat{\score} - \score^\mathcal{F}\|^2_{L^2(\mu)} + \lambda \|\nabla^{s-1} (\hat{\score} - \score^\mathcal{F})\|_{L^2(\mathcal{T}^d)}^2 + \delta L^{\lambda}(\score^\mathcal{F}; \hat{\score}-\score^\mathcal{F}),
\end{aligned}
\end{multline*}
where $\delta L^{\lambda}$ is the G\^{a}teaux derivative of $L^{\lambda}$ whose value is 
\begin{multline*}
\delta L^{\lambda}(\score^\mathcal{F}; \hat{\score}-\score^\mathcal{F}) = \\
2 \left(\int_{\mathcal{T}^d} \left<\score^\mathcal{F} - \score^\star, \hat{\score}-\score^\mathcal{F} \right>  d\mu + \lambda \left<\nabla^{s-1} \score^\mathcal{F}, \nabla^{s-1} (\hat{\score}-\score^\mathcal{F})  \right>_{L^2(\mathcal{T}^d)}\right).
\end{multline*}
Then, since $\score^\mathcal{F}$ minimizes $L^\lambda$ and $\mathcal{F}$ is convex, we have $\delta L^{\lambda}(\score^\mathcal{F}; \hat{\score}-\score^\mathcal{F}) \geq 0$ and thus 
\[
\|\hat{\score} - \score^\mathcal{F}\|_{L^2(\mu)}^2  + \lambda \|\nabla^{s-1} (\hat{\score} - \score^\mathcal{F})\|_{L^2(\mathcal{T}^d)}^2  \leq L^{\lambda}(\hat{\score}) - L^{\lambda}(\score^\mathcal{F}). 
\]
Now, since $\widehat{L}^\lambda(\score^\mathcal{F})  - \widehat{L}^\lambda(\hat{\score}) \geq 0$, we can add this term to the right-hand side of the previous inequality to obtain 
\begin{multline*}
\|\hat{\score} - \score^\mathcal{F}\|_{L^2(\mu)}^2  + \lambda \|\nabla^{s-1} (\hat{\score} - \score^\mathcal{F})\|_{L^2(\mathcal{T}^d)}^2\\
\begin{aligned}
& \leq (L^{\lambda}(\hat{\score}) - L^{\lambda}(\score^\mathcal{F})) - (\widehat{L}^{\lambda}(\hat{\score}) - \widehat{L}^{\lambda}(\score^\mathcal{F})) \\
&= L(\hat{\score}) -L(\score^\mathcal{F}) -(\widehat L(\hat{\score}) - \widehat L(\score^\mathcal{F})) \\
& \leq \sup_{g\in \mathcal{F} } (L(g) - \widehat{L}(g)) - (
L(\score^\mathcal{F}) - \widehat{L}(\score^\mathcal{F}))\\
& \leq \sup_{g \in \mathcal{F} } \frac{1}{n}  \sum_{i=1}^n \Big(\E[l_g(X_i)] - l_g(X_i) + l_{\score^\mathcal{F}}(X_i) - \E[l_{\score^\mathcal{F}}(X_i)] \Big)\\
 &\leq  \sup_{g \in \mathcal{F} } \frac{1}{n}  \sum_{i=1}^n  \left(\tilde{l}_{\score^\mathcal{F}}(X_i) - \tilde{l}_g(X_i)\right),
\end{aligned}
\end{multline*}
which yields \eqref{but1}.\medbreak

\noindent \textbf{Step 2:} Let us upper-bound the right-hand side of \eqref{but1}. Since $\mathcal{F} \subset \mathcal{F}_{s,R,\window,\ell}$, it is possible to find a constant $C_0$ such that 
\[
\frac{1}{C_0} \sup_{g\in \mathcal{F} }\|\tilde{l}_g - \tilde{l}_{\score^\mathcal{F}}\|_{L^\infty(\mathcal{T}^d)}\leq 1.
\]
For $r>0$, let us set $\mathcal{F}_r := \mathcal{F} \cap B_{W^{1,2}(\mu)}(\score^\mathcal{F}, C_0 r)$ and let us define
\begin{equation}
S := \frac{1}{\sqrt{n}} \sup_{g\in \mathcal{F}_r} \frac{1}{C_0} \left|\sum_{i=1}^n \tilde{l}_{\score^\mathcal{F}}(X_i) - \tilde{l}_g(X_i) \right|,
\end{equation}
which is related to the right-hand side of \eqref{but1}, as shall be seen later. Our purpose in this Step 2 is to bound $\E[S]$. Using the Bernstein inequality is natural to achieve this, but we follow here the chaining techniques in \cite{Massart} that will provide tighter bounds.\medbreak

Chaining arguments require estimates on log-Laplace functionals (see conditions of Lemma 13.1 and Corollary 13.2 \cite{Massart}) that are satisfied when using Rademacher random variables. 
Let us therefore consider $\epsilon_1, \dots, \epsilon_n$ i.i.d. Rademacher random variables independent of $X_1, \dots, X_n$. By \cite[Lemma 11.4]{Massart}, we have
\begin{equation}\label{etape7}
\E [S] \leq 2\E \left[\sup_{g\in \mathcal{F}_r} \frac{1}{C_0\sqrt{n}} \left|\sum_{i=1}^n \epsilon_i (\tilde{l}_{\score^\mathcal{F}}(X_i) - \tilde{l}_g(X_i))\right|\right].
\end{equation}
Notice that for any $g_1, g_2 \in \mathcal{F}_r$,
\begin{align*}
\frac{1}{C_0^2}\E\left[(\tilde{l}_{g_1} - \tilde{l}_
{g_2})^2(X_1)\right]& = \frac{1}{C_0^2}\Var\left[(2  \div (g_1-g_2) + \| g_1\|^2 - \| g_2\|^2 )(X_1)\right] \\
& \leq  C r^2,
\end{align*}
thus we can define $\delta_n$ as 
\begin{equation}\label{def:deltan}
\delta_n^2 = \max\Big(  \sup_{g\in \mathcal{F}_{r}} \frac{1}{C_0^2 n}\sum_{i=1}^n (\tilde{l}_{\score^\mathcal{F}}(X_i) - \tilde{l}_g(X_i))^2,\  r^2\Big).
\end{equation}
Now following the beginning of the proof of \cite[Lemma 13.5]{Massart}
we obtain an upper-bound for the right-hand side of \eqref{etape7}: 
\[
\E[S] \leq C \E\left[\sum_{j=1}^\infty \delta_n 2^{-j} \sqrt{H(2^{-(j+1)} \delta_n)}\right],
\]
for some positive constant $C > 0$. Let us introduce for $u>0$ the entropy number $H(u)$ of the metric space $\mathcal{F}_r$ equipped with the distance
\begin{equation}\label{def:distance_ltilde}
\forall g_1, g_2 \in \mathcal{F}_r,  d(g_1, g_2)^2 = \frac{1}{C_0^2} \|\tilde{l}_{g_1} - \tilde{l}_{g_2}\|^2_{L^{\infty}(\mathcal{T}_d)} \geq \frac{1}{C_0^2 n} \sum_{i=1}^n (\tilde{l}_{g_1} - \tilde{l}_
{g_2})^2(X_i). 
\end{equation}
The metric entropy $H(u)$ is defined as the log of the $u$-packing number, which is the maximal number of disjoint balls of radius $u/2$ that can be packed into $\mathcal{F}_r$. From its definition, $H(\cdot)$ is a non-increasing function with respect to $u$. We also define the following quantity
\begin{equation}\label{def:D}
D = \sum_{j=1}^\infty 2^{-j} \sqrt{H(2^{-(j+1)} r)}.
\end{equation}
Since $\delta_n \geq r$ by definition, and since 
$H(\cdot)$ is a non-increasing function, we have: 
\begin{equation}\label{etape9}
\E [S]\leq C \E[\delta_n^2]^{1/2} D.
\end{equation}
Let us find an upper-bound for $\E[\delta_n^2]$.  From definition \eqref{def:deltan},
\begin{equation}\label{deltan}
\E[\delta_n^2]  \leq r^2 + \frac{1}{n} \E\left[\sup_{g\in \mathcal{F}_r} \frac{1}{C_0^2} \sum_{i=1}^n \big(\tilde{l}_{\score^\mathcal{F}}(X_i) - \tilde{l}_g(X_i)\big)^2\right].
\end{equation}
To control the l.h.s of \eqref{deltan}, we use Theorem 11.8 \cite{Massart} and to this end we set using their notation the \textit{weak variance} $\Sigma^2$ as
\[
\Sigma^2 = \E\left[\sup_{g\in \mathcal{F}_r} \frac{1}{C_0^2} \sum_{i=1}^n \big(\tilde{l}_{\score^\mathcal{F}}(X_i) - \tilde{l}_g(X_i)\big)^2\right],
\]
which entails that
\[
\E[\delta_n^2] \leq \frac 1 n (n r^2 + \Sigma^2).
\]
Theorem 11.8 \cite{Massart} gives an upper-bound for $n r^2 + \Sigma^2$ which yields 
\begin{align}
\E[\delta_n^2]  &\leq \frac 1 n \left (2nr^2  + C \E\Bigg[\sup_{g\in \mathcal{F}_r} \frac{1}{C_0} \sum_{i=1}^n \big(\tilde{l}_g-\tilde{l}_{\score^\mathcal{F}}\big)(X_i)\Bigg] \right)  \nonumber \\
 &\leq 2 r^2 + \frac{C}{\sqrt{n}} \E [S]. \label{etape10}
\end{align}

From \eqref{etape9} and \eqref{etape10}, we obtain
\[
\E[S] \leq C D \sqrt{r^2 + \frac{\E[S]}{\sqrt{n }}}.
\]
Solving this quadratic inequality for $\E[S]$, we obtain 
\[
\E[S] \leq \frac{C^2 D^2}{ 2\sqrt{n}} \left(1 + \sqrt{1 + \frac{4 r^2 n}{C^2 D^2}} \right).
\]
In particular, for $n$ sufficiently large, we have $D^2 \leq Cnr^2$ and
\begin{equation}\label{etape11}
\E[S] \leq C r D.
\end{equation}
Let us now bound $D$ by establishing an upper-bound on the entropy $H(\cdot)$. First, remark that the distance defined in \eqref{def:distance_ltilde} verifies for all $g_1,\ g_2\in \mathcal{F}_{s, R, \window, \ell} $,
\begin{align*}
d(g_1, g_2)^2 & \leq C \|\tilde{l}_{g_1} - \tilde{l}_{g_2}\|^2_{L^\infty(\mathcal{T}^d)} 
 \leq C \|g_1 - g_2\|^2_{W^{1,\infty}({\mathcal{T}^d})}.
\end{align*}
Then, since $\mathcal{F} \subset \mathcal{F}_{s, R, \window, \ell}$, we have 
\[
\forall g \in \mathcal{F}, \|\tilde{l}_g\|_{W^{s-2+\ell,\infty}(\mathcal{T}^d)} \leq \frac{C}{\window^{\ell}}.
\]
Therefore, by \cite[Theorem 2]{edmunds}, 
\[
H(u) \leq C (h^{\ell} u)^{-\frac{d}{\ell+s-2}}
\]
and thus plugging this in \eqref{def:D}, and since $2(s+\ell-2) > d$,
\[
D \leq C (h^{\ell} r)^{\frac{-d}{2(s+\ell-2)}} \sum_{j=1}^\infty  2^{j (\frac{d}{2(s+\ell-2)}-1)} \leq C (h^{\ell} r)^{\frac{-d}{2(s+\ell-2)}}.
\]
This upper-bound on $D$ and \eqref{etape11} yield
\begin{equation}
\E[S] \leq C h^{\frac{-d\ell}{2(s+\ell-2)}} r^{\frac{2(s + \ell -2) -d}{2(s+\ell-2)}}.
\end{equation}

\noindent \textbf{Step 3:}
Hence, we can use Theorem 13.19 \cite{Massart} with, following their theorem notation,
\begin{itemize}
    \item $\sigma(g)^2 = C \|g - \score^\mathcal{F}\|^2_{W^{1,2}(\mu)} $;
    \item $\rho : u \rightarrow u$;
    \item $\psi : r \rightarrow C h^{\frac{-d\ell}{2(s+\ell-2)}} r^{\frac{2(s + \ell -2) -d}{2(s+\ell-2)}}$; 
    \item $\epsilon = C_0 \lambda^{\frac{1}{s-1}}$
\end{itemize} 
for some $C_0 > 0$ to obtain that, with probability $1 - 2e^{-2x}$, 
\begin{align*}
\|\hat{\score} - \score^\mathcal{F}\|_{L^2(\mu)}^2  + \lambda \|\nabla^{s-1} (\hat{\score} - \score^\mathcal{F})\|_{L^2(\mathcal{T}^d)}^2 & \leq \frac{1}{n} \sum_{i=1}^n \tilde{l}_{\score^\mathcal{F}}(X_i) - \tilde{l}_{\hat{\score}}(X_i)\\
& \leq C_0 \lambda^{\frac{1}{s-1}} \left(\sigma(\hat{\score})^2 + r_\star^2 + \frac{x}{n}\right),
\end{align*}
where $r_\star$ is the solution of 
\[
    \sqrt{n} r^2 = C(C_0) \lambda^{-\frac{1}{{s-1}}}  h^{\frac{-d\ell}{2(s+\ell-2)}} r^{\frac{2(s + \ell -2) -d}{2(s+\ell-2)}}
\]
(see Exercise 13.42 \cite{Massart}). This solution can be computed to obtain
\begin{equation}\label{eq:rstar-solve}
r_\star = C(C_0) \lambda^{-\frac{2(s+\ell-2)}{(s-1)(d+2(s+\ell-2))}} h^{-\frac{d\ell}{d+2(s+\ell-2)}} n^{-\frac{s+\ell-2}{d+2(s+\ell-2)}}.
\end{equation}

We are thus left with bounding $\sigma(\hat \score)^2$. Since the density of $\mu$ is bounded from above, using the Gagliardo-Nirenberg inequality (see Theorem 3.70 \cite{aubin}) yields
\begin{align*}
\lambda^{\frac{1}{s-1}} \|\nabla (\hat{\score} - \score^\mathcal{F})&\|_{L^2(\mu)}^2  \\
& \leq \|f^\mu\|_{L^\infty(\mathcal{T}^d)} \lambda^{\frac{1}{s-1}} \|\nabla (\hat{\score} - \score^\mathcal{F})\|_{L^2(\mathcal{T}^d)}^2 \\
& \leq \|f^\mu\|_{L^\infty(\mathcal{T}^d)} \lambda^{\frac{1}{s-1}} \|\hat{\score} - \score^\mathcal{F}\|^{2-\frac{2}{s-1}}_{L^2(\mathcal{T}^d)}  \|\nabla^{s-1} (\hat{\score} - \score^\mathcal{F})\|^{\frac{2}{s-1}}_{L^2(\mathcal{T}^d)} \\
& \leq \|f^\mu\|_{L^\infty(\mathcal{T}^d)} \lambda^{\frac{1}{s-1}} \|\hat{\score} - \score^\mathcal{F}\|^{\frac{2(s-2)}{s-1}}_{L^2(\mathcal{T}^d)}  \|\nabla^{s-1} (\hat{\score} - \score^\mathcal{F})\|^{\frac{2}{s-1}}_{L^2(\mathcal{T}^d)}.
\end{align*}
Since $\frac{s-2}{s-1} + \frac{1}{s-1} = 1$, we can use Young's inequality for products to obtain 
\begin{multline*}
\lambda^{\frac{1}{s-1}} \|\nabla (\hat{\score} - \score^\mathcal{F})\|_{L^2(\mu)}^2 \leq \\
\|f^\mu\|_{L^\infty(\mathcal{T}^d)} \left(\frac{s-2}{s-1} \|\hat{\score} - \score^\mathcal{F}\|^{2}_{L^2(\mathcal{T}^d)} + \frac{\lambda}{s-1} \|\nabla^{s-1} (\hat{\score} - \score^\mathcal{F})\|^{2}_{L^2(\mathcal{T}^d)}\right).
\end{multline*}
Finally, since $\mu$ has a density bounded from below on $\mathcal{T}^d$, 
\begin{multline*}
\lambda^{\frac{1}{s-1}} \|\nabla (\hat{\score} - \score^\mathcal{F})\|_{L^2(\mu)}^2 
\leq \|f^\mu\|_{L^\infty(\mathcal{T}^d)}  \bigg(\frac{s-2}{s-1} \|\hat{\score} - \score^\mathcal{F}\|^{2}_{L^2{(\mu)}} \left\|\frac{1}{f^\mu}\right\|_{L^\infty(\mathcal{T}^d)}   \\+ \frac{\lambda}{s-1} \|\nabla^{s-1} (\hat{\score} - \score^\mathcal{F})\|^{2}_{L^2(\mathcal{T}^d)}\bigg).
\end{multline*}
Therefore, we obtained that 
\[
\lambda^{\frac{1}{s-1}} \sigma(\hat{\score})^2 \leq C \left( \| \hat{\score} -\score^\mathcal{F}\|^2_{L^2(\mu)} + \lambda\|\nabla^{s-1} (\hat{\score} - \score^\mathcal{F})\|^{2}_{L^2(\mathcal{T}^d)}\right)
\]
and thus, with probability greater than $1-2e^{-2x}$,
\begin{multline*}
\|\hat{\score} - \score^\mathcal{F}\|_{L^2(\mu)}^2  + \lambda \|\nabla^{s-1} (\hat{\score} - \score^\mathcal{F})\|_{L^2(\mathcal{T}^d)}^2 \leq  \\
C C_0 \left( \| \hat{\score} -\score^\mathcal{F}\|^2_{L^2(\mu)} + \lambda\|\nabla^{s-1} (\hat{\score} - \score^\mathcal{F})\|^{2}_{L^2(\mathcal{T}^d)}+ \lambda^{\frac{1}{s-1}}(r_\star^2 + \frac{x}{n}) \right).
\end{multline*}
As the constant $C C_0$ above can be made smaller than $\frac{1}{2}$ by taking $C_0$ sufficiently small, this entails that
\[
\|\hat{\score} - \score^\mathcal{F}\|_{L^2(\mu)}^2 \leq C\lambda^{\frac{1}{s-1}}(r_\star^2 + \frac{x}{n}) 
\]
which after integration with respect to $x$ gives 
\[
\E[\|\hat{\score} - \score^\mathcal{F}\|_{L^2(\mu)}^2] \leq C \lambda^{\frac{1}{s-1}} (r_\star^2 + \frac{1}{n}).
\]
Finally, plugging \eqref{eq:rstar-solve} in the above inequality yields the result.
\end{proof}


\section{Score estimation with diffusion}\label{sec:avecdiffusion}
\subsection{Setting}

Let us now deal with the problem of estimating the score function for diffusion models. Let $\mu$ be a probability measure on $\R^D$ and, for $t > 0$, let $\mu_t$ be the distribution of the random variable $\overrightarrow{X}_t $ defined in \eqref{eq:OUconvo}. This measure admits a density $f^{\mu_t}$ for any $t>0$, see \eqref{eq:densite_h-1}, and we recall the score function we want to estimate is $\score^\star_t = \nabla \log f^{\mu_t}$. For our estimation, we consider assumptions similar to the ones used in \cite{Divol2022} for measure estimation on a manifold. 

\begin{hyp}
\label{hyp:diffusion}
The probability measure $\mu$ is a measure supported on a $d$-dimensional submanifold $\mathcal{M}$ of $\mathbb{R}^D$ such that 
\begin{itemize}
\item $\mathcal{M}$ is compact and without boundary;
\item the reach of $\mathcal{M}$ is strictly larger than $\tau_{\min}$;
\item there exists a positive integer $k$ such that $\mathcal{M}$ is $k$-smooth: there exists $L > 0$ such that for any $x \in \mathcal{M}$, the orthogonal projection on the tangent space $\tangentspace$ of $\mathcal{M}$ at $x$, $\projection_x : \mathcal{M} \rightarrow \tangentspace$, is a local diffeomorphism at $x$ whose inverse $\Psi_x$ defined on $B_{\tangentspace}(x, r)$, with $r = \frac{\min(\tau_{\min}, L)}{4}$, is $k$-times differentiable with bounded derivatives.
\end{itemize}
Furthermore, $\mu$ admits a density $\density^\mu$ with respect to the volume measure of $\mathcal{M}$ such that 
\begin{itemize}
\item $\density^\mu$ is bounded from below by $\density_{\min} > 0$;
\item there exists $1 \leq  s \leq k-3$ such that $\density^\mu$ is $s$-times differentiable on $\mathcal{M}$ and its derivatives are bounded.
\end{itemize}
\end{hyp}

\subsection{Score function estimator}
\label{sec:estimatordef}
Let $t > 0$. It is well known (see e.g. \cite{Vincent2011}) that the score-function $\score^\star_t$ can be defined as 
\begin{equation}
\label{eq:conditionalscoredef}
\score^\star_t(\overrightarrow{X}_t) = -\E\left[\frac{Z_t}{\sqrt{1-e^{-2t}}} \mid \overrightarrow{X}_t \right]
\end{equation}
and, as in Section~\ref{sec:nondiffestimator}, it minimizes the following implicit score matching loss function, defined for any vector field $g \in W^{1,2}(\mu_t)$ as
\begin{align*}
L_t(g) &  =  \int_{\R^D} 2 \div g + \|g\|^2 \, d\mu_t \\
 &=  \E\Big[2 \div g (\overrightarrow{X}_t)+ \|g  (\overrightarrow{X}_t)\|^2 \Big]\\
& =  \E\Big[ \E\left[2 \div g (\overrightarrow{X}_t)+ \|g  (\overrightarrow{X}_t)\|^2 \mid X\right]\Big] \\
&= \E\left[ P_t(2 \div g+ \|g\|^2)(X)\right]\\
& = \int_{\mathcal{M}} P_t(2 \div g + \|g\|^2) \, d\mu ,
\end{align*}
where $(P_t)_{t \geq 0}$ denotes the Ornstein-Uhlenbeck semigroup:
\[
\forall \phi \in L^{1}(\mu_t), \forall x \in \R^D, P_t \phi(x) = \E[\phi(\overrightarrow{X}_t) \mid X = x].
\]
Therefore, by taking
\[
l_{g,t} = P_t(2\div g + \|g\|^2),
\]
we can rewrite the loss as 
\[
L_t(g) = \int_{\mathcal{M}} l_{g,t} \, d\mu = \E[l_{g,t}(X)]. 
\]

Given a sample $X_1, \dots, X_n$ of i.i.d. random variables drawn from $\mu$, we can define an empirical version of our loss
\[
\widehat{L}_t(g) = \frac{1}{n} \sum_{i=1}^n l_{g,t}(X_i)
\]
We consider the empirical minimizer over a hypothesis class $\mathcal{F} \subset W^{1,\infty}(\R^D)$ defined by 
\begin{equation}\label{eq:hat-score}
\hat{\score}_t \in \argmin_{g\in \mathcal{F}} \widehat{L}_t(g).
\end{equation}
Remark that, compared to the non-diffusion case, we do not need to add a penalization on the higher-order derivatives of the function considered thanks to the smoothing properties of the semigroup $(P_t)_{t \geq 0}$ (see Lemma \ref{lem:integrationByParts}). Also note that, in practice, score functions are estimated by minimizing the equivalent denoising score-matching loss \cite{Vincent2011}. 

As in the non-diffusion case, we define an appropriately sized hypothesis class by letting $R \geq 0, h > 0, \ell \in \mathbb{N}$ and 
\begin{multline*}
\mathcal{F}_{t, s, R, h, \ell}=\Big\{g\in W^{s+1+\ell, \infty}(\R^D) \ \big|\  \left\|g\right\|_{W^{s+1,\infty}(\R^D)} \leq \frac{R}{1-e^{-2t}} ,\ \\
\forall i \in \{0,\dots,\ell\}, \left\|g\right\|_{W^{s+1+i,\infty}(\R^D)} \leq \frac{R}{(1-e^{-2t})\window^{i}}\Big\}.
\end{multline*}
Such a class of vector fields is built by remarking that, for a sufficiently large $R > 0$,
\[
\max_{0 \leq i \leq s+1}  \left\| \nabla^i \score^\star_t(x)\right\| \leq \frac{R}{1-e^{-2t}}
\]
for all $x \in \R^D$ sufficiently close to $\mathcal{M}$. Thus $\mathcal{F}_{t, s, R, h, \ell}$ contains a kernel-smoothed approximation with bandwidth $\window$ of a clipped version of $\score^\star_t$, see Section~\ref{sec:biasdiffusion} for more details. As mentioned in the introduction, our hypothesis class of vector fields belongs to a Sobolev ball rather than corresponding to a neural network architecture. Another important difference with works such as \cite{oko, Stephanovitch, zhang2026, fu2026} is that we do not assume that candidate vector fields $g$ verify $\|g\|_{L^\infty(\R^D)} \leq \frac{R}{\sqrt{1-e^{-2t}}}$ which would greatly simplify the proofs. Such a condition could be obtained through clipping around the manifold $\mathcal{M}$; however such an aggressive clipping would make the higher-order derivatives explode faster than the current $\frac{1}{t}$. Remarking that the condition we would truly require in our proof is \[\|P_t (\|g\|^2)\|_{L^\infty(\mathcal{M})} \leq \frac{R^2}{1-e^{-2t}},\] which is verified by $\score_t^\star$, one may want to add this condition on the class $\mathcal{F}_{t, s, R, h, \ell}$ but this would require knowing, or at least directly estimating, the manifold $\mathcal{M}$. We are now ready to state our estimation result. 
\begin{theorem}
\label{thm:maindiffusion}
Suppose Assumption \ref{hyp:diffusion} is verified and let $m \in \N$ such that $2(s+2+m) > d$. If $\mathcal{F} = \mathcal{F}_{t, s, R, \window, \ell}$, $R$ sufficiently large and $\ell$ such that $2(s+2+\ell) > d$ then there exist $C,C' > 0$ independent of $n, t$ and $\window$ such that, if $1-e^{-2t} \leq \frac{C'}{\log(n)}$,  
\[
\E\left[\|\hat{\score}_t - \score_t^\star \|^2_{L^2(\mu_t)}\right] \leq \\
\frac{C}{(1-e^{-2t})^2} \left( \window^{2(s+1)} + r^2  + \log(n)^{2(s+1)} n^{-\frac{2(s+1)}{d}} \right),
\]
where 
\[
r^2 = \begin{cases}
\left(n^{2(s+2+\ell)} (1-e^{-2t})^{d-2(s+2+\ell)} \window^{2\ell d}\right)^{-\frac{1}{2(s+2+\ell)+d}} \text{ if $\sqrt{1-e^{-2t}} \leq h$} \\
   \left(n^{2(s+2+m)} (1-e^{-2t})^{d(m+1)-2(s+2+m)} \right)^{-\frac{1}{2(s+2+m)+d}}  \text{ otherwise}
\end{cases},
\]
provided $r^2 \leq 1-e^{-2t}$.

\end{theorem}

\begin{remark}
This result exhibits two regimes. Whenever $\sqrt{1-e^{-2t}} \leq h$, we use the regularity of the ``kernel-like" smoothing of our hypothesis class. On the other hand, as soon as $\sqrt{1-e^{-2t}} \geq h$, this need for additional smoothing vanishes thanks to the semigroup $(P_t)_{t \geq 0}$. Indeed, at small times, applying $(P_t)_{t \geq 0}$ to a function is nearly equivalent to smoothing this function with a Gaussian kernel with window $\sqrt{1-e^{-2t}}$.
\end{remark}

\begin{remark}
\label{rem:large-times}
The condition $1-e^{-2t} \leq \frac{C}{\log(n)}$ in Theorem~\ref{thm:maindiffusion} is required because we only study the regularity of the score function around the manifold $\mathcal{M}$ where most of the measure $\mu_t$ is distributed when $t$ is small. A similar result can be obtained with $s = -1$ for any $t$, using Lemma~\ref{lem:simple-bias} to replace Proposition~\ref{pro:final-bias} to deal with the bias term and taking $\mathcal{F} = \mathcal{F}_R = \{g \mid \|g + \frac{id}{1-e^{-2t}}\|_{L^{\infty}(\R^D)} \leq \frac{R}{1-e^{-2t}} \}$.
\end{remark}

This result is obtained in a similar fashion to Theorem~\ref{rate-score-tore-plat}. However, the proof is more technical due to the manifold assumption and the need to exploit the smoothing properties of the semigroup $(P_t)_{t \geq 0}$. Therefore, we will give the outline of the proof below while delaying the technical parts to Section~\ref{sec:proofs}. The first step to prove the theorem is to derive a bias-variance trade-off result for our estimator which is provided by the following lemma. 
\begin{lemma}
\label{lem:bias-variance}
For any $t > 0$, $g_0 \in \mathcal{F}$,
\[
\|\hat{\score}_t - \score^\star_t\|^2_{L^2(\mu_t)} \leq \|g_0 - \score^\star_t\|^2_{L^2(\mu_t)} + (\widehat{L}_t(g_0) - \widehat{L}_t(\hat{\score}_t)) -  (L_t(g_0) - L_t(\hat{\score}_t)) . 
\]
\end{lemma}

As usual, we will face a bias-variance trade-off: the larger $\mathcal{F}$ is, the smaller we can make the bias term $\|g_0 - \score_t^\star\|^2_{L^2(\mu_t)}$ but also the larger the variance term $ (\widehat{L}_t(g_0) - \widehat{L}_t(\hat{\score}_t)) -  (L_t(g_0) - L_t(\hat{\score}_t))$. Theorem~\ref{thm:maindiffusion} is then obtained by combining a bound on the variance term from Proposition~\ref{pro:final-variance} and taking $g_0 = \score_{t,h}$ from Proposition~\ref{pro:final-bias} when $h^{2(s+1)} \leq 1-e^{-2t}$ and $g_0 = 0$ otherwise to deal with the bias term.

\subsection{Bias term} 
\label{sec:biasdiffusion}
In order to deal with the bias term, our first objective is to study the regularity of the score function. We first start by proving a crude result which will be relevant for ``large" times. 
\begin{lemma}
\label{lem:simple-bias}
For any $t > 0$ and any $x \in \R^D$, 
\[
\left\|\score_t^\star(x) + \frac{x}{1-e^{-2t}} \right\| \leq \frac{e^{-t} \sup_{y \in \mathcal{M}} \|y\|}{1-e^{-2t}}.
\]
\end{lemma}

\begin{proof}
By Equation~(\ref{eq:conditionalscoredef}),  
\begin{align*}
\score_t^\star(\overrightarrow{X}_t) + \frac{\overrightarrow{X}_t}{1-e^{-2t}} & = \E\left[- \frac{Z_t}{\sqrt{1-e^{-2t}}} + \frac{e^{-t}X + \sqrt{1-e^{-2t}}Z_t}{1-e^{-2t}} \mid \overrightarrow{X}_t \right] \\
&= \E\left[\frac{e^{-t} X}{1-e^{-2t}} \mid \overrightarrow{X}_t \right],
\end{align*}
concluding the proof.
\end{proof}

For ``small times", as we are only interested in estimating the score function in the $L^2(\mu_t)$ sense and since most of the mass of $\mu_t$ is distributed near the manifold $\mathcal{M}_t = \{e^{-t} x \mid x \in \mathcal{M}\}$ it is sufficient to study the regularity of the score function around it. For $\epsilon > 0$, let $\mathcal{M}_t^{\epsilon} = \{x \in \R^D \mid \exists y \in \mathcal{M}_t, \|y-x\| < e^{-t} \epsilon\}$. As the score function $\score^\star_t$ is the gradient of the logarithm of the density $f^{\mu_t}$ of $\overrightarrow{X}_t$, we can study the regularity of $f^{\mu_t}$ to understand the regularity of $\score^\star_t$. Let us start by decomposing the density $f^{\mu_t}$. For this we introduce the measure $\nu_t$ of $e^{-t}X$ and the orthogonal projection on $\mathcal{M}_t$, denoted by $p_t$.
\begin{lemma}
\label{lem:density-decomposition}
Let $x \in \mathcal{M}_t^{\tau_{\min}}$ and $t < 1$. We have
\[
f^{\mu_t}(x) = f^\perp_t(x) f^\|_t(x),
\]
where
\[
f^\perp_t(x) = \frac{1}{\left(2\pi (1-e^{-2t})\right)^{(D-d)/2}}    e^{-\frac{\|x-p_t(x)\|^2}{2(1-e^{-2t})}}
\]
and
\[
f^\|_t(x) = \frac{1}{\left(2\pi (1-e^{-2t})\right)^{d/2}} \int_{\mathcal{M}_t} e^{-\frac{\|p_t(x)-y\|^2 + 2 \left<x - p_t(x) ,p_t(x) - y \right>}{2(1-e^{-2t})}} d\nu_t(y)
\]
\end{lemma}

\begin{proof}
From \eqref{eq:densite_h-1},
\[
f^{\mu_t}(x) = \frac{1}{\left(2\pi (1-e^{-2t})\right)^{D/2}} \int_{\mathcal{M}_t} e^{-\frac{\|x -y\|^2}{2(1-e^{-2t})}} d\nu_t(y),
\]
the result directly follows from 
\[
\|x-x_0 e^{-t}\|^2= \|x-p_t(x)\|^2 + \|p_t(x)-x_0 e^{-t}\|^2 + 2\left< x-p_t(x),\ p_t(x)-x_0 e^{-t} \right>.
\]
\end{proof}

Let us consider $\epsilon=\tau_{\min}$, the lower bound on the reach introduced in Assumption \ref{hyp:diffusion}, and let $x \in \mathcal{M}_t^{\tau_{\min}}$. By Theorem C from \cite{Leobacher2021}, $\nabla p_t(x) (x - p_t(x)) = 0$. Thus, the score function satisfies
\begin{equation}
\label{eq:scoreshape}
    \score^\star_t(x) = - \frac{x - p_t(x)}{1-e^{-2t}} + \nabla \log f^\|_t(x).
\end{equation}
The first term in \eqref{eq:scoreshape} is thus orthogonal to $\mathcal{M}_t$ and explodes as $t$ goes to $0$. Since $p_t:\mathcal{M}^{\tau_{\min}}_t \rightarrow \mathcal{M}_t$ is $k-1$ times differentiable \cite{Leobacher2021}, this first component is also $k-1$ times differentiable with derivatives of order $\frac{1}{t}$. We then show, by a succession of integrations by parts on the manifold, that the regularity of the second component is determined by the regularity of $\mu$.
\begin{lemma}
\label{lem:parallel-regularity}
There exist $\tau < \tau_{\min}$ and $R > 0$ such that for all $0<t<1$, for all $x \in \mathcal{M}_t^{\tau/2}$, and for all $i\leq s$,
\[
 \left\|\nabla^i \log f^\|_t(x)\right\| \leq R.
\]
Additionally, for $j\leq 2$,
\[
 \|\nabla^{s+j} \log f^\|_t(x) \| \leq \frac{R}{(1-e^{-2t})^{j/2}}.
\]
\end{lemma}
This result only deals with the regularity of $\score^\star_t$ close to the manifold $\mathcal{M}_t$. Fortunately, as we only need to properly estimate $\score^\star_t$ in the $L^2(\mu_t)$ sense, it is sufficient to properly estimate the score function around $\mathcal{M}_t$ at small times.
\begin{lemma}
\label{lem:cutscore}
There exist $T < 1$ and $R >0$ such that, for any $t < T$, there exists a function $\tilde{\score}_t$ such that
\begin{itemize}
    \item $\|\tilde{\score}_t\|_{W^{s+1,\infty}} \leq \frac{R}{1-e^{-2t}}$;
    \item $\|\tilde{\score}_t - \score_t^\star\|^2_{L^2(\mu_t)} \leq \frac{C }{(1-e^{-2t})^{2 + D/2}} e^{-\frac{C}{1-e^{-2t}}}$.
\end{itemize}
\end{lemma}
Finally, smoothing $\tilde{\score}_t$ with an adequate kernel function as in the proof of Proposition~\ref{prop:biais} provides us with a function in $\mathcal{F}_{t, s, R, h, \ell}$ that is sufficiently close to $\score_t^\star$.
\begin{proposition}
\label{pro:final-bias}
There exist $T < 1, R> 0$ and $C$ independent of $n$ and $t$ such that for any $\window > 0, t < T$ there exists $\score_{t,h} \in \mathcal{F}_{t, s, R, h, \ell}$ satisfying
\[
\|\score_{t,h} - \score^\star_t\|^2_{L^2(\mu_t)} \leq \frac{C\window^{2(s+1)}}{(1-e^{-2t})^2} + \frac{C }{(1-e^{-2t})^{2 + D/2}} e^{-\frac{C}{1-e^{-2t}}}.
\]
Furthermore, provided $h^{2(s+1)} \leq C(1 - e^{-2t})$, we have 
\[
\left\|P_t(\|\score_{t,h}\|^2)\right\|_{L^\infty(\mathcal{M})} \leq \frac{C}{1-e^{-2t}}.
\]
\end{proposition}

\subsection{Variance term} 
\label{sec:variancediffusion}
Let $m \in \N$ such that $2(s+2+m) > d$ and suppose $\mathcal{F} \subset \mathcal{F}_{t, s, R, h, \ell}$ for some $\ell$ such that $2(s+2+\ell) > d$. 
In this section, we denote by $C$ a generic constant independent of $t,n$ and $\window$. As in the non-diffusion case, we obtain the result through a chaining and reweighting argument. However, we need to decompose the loss difference into two terms of different magnitudes and treat them separately. More precisely, for any $g, g_0 \in \mathcal{F}_{t,s,R,h,\ell}$, we have 
\begin{align*}
l_{g,t} - l_{g_0,t}  & =  P_t\left(2 \div(g - g_0) + \|g\|^2 - \|g_0\|^2 \right) \\
& = P_t\left(2 \div(g - g_0) + \left<g - g_0, g + g_0\right> \right) \\
& = 2 P_t\left( \div(g - g_0) + \left<g - g_0,  g_0\right> \right) + P_t \left( \|g - g_0\|^2\right) \\
& = 2 l^{lin}_{g, t} + l^{quad}_{g, t}
\end{align*}
with
\[
l^{lin}_{g, t} = P_t\left( \div(g - g_0) + \left<g - g_0,  g_0\right> \right)
\]
and 
\[
l^{quad}_{g,t} = P_t \left( \|g - g_0\|^2\right).
\]
Remark that, if $P_t (\|g_0\|^2) \leq \frac{C}{1-e^{-2t}}$ (a condition verified by $\score_t^\star$) then, by Lemma~\ref{lem:integrationByParts}, we have
\begin{equation}
\label{eq:Llinbound}
\|l^{lin}_{g, t}\| \leq \frac{C}{\sqrt{1-e^{-2t}}} \sqrt{P_t \|g - g_0\|^2} \leq \frac{C}{(1-e^{-2t})^{3/2}}
\end{equation}
while
\begin{equation}
\label{eq:Lquadbound}
l^{quad}_{g,t} = P_t \left( \|g - g_0\|^2\right) \leq \frac{C}{(1-e^{-2t})^2},
\end{equation}
hence the difference in magnitudes between the two terms. By definition, we have
\begin{multline*}
(\widehat{L}_t(g_0) - \widehat{L}_t(g)) -  (L_t(g_0) - L_t(g)) \\ =
-\frac{1}{n} \sum_{i=1}^n 2 \left(l^{lin}_{g, t} (X_i) - \E[l^{lin}_{g, t}(X_i)]\right) + \left(l^{quad}_{g,t} (X_i) - \E[l^{quad}_{g,t}(X_i)]\right). 
\end{multline*}
We deal with the linear component of the loss using arguments similar to those we used non-diffusion case while exploiting the regularization properties of the semigroup $(P_t)_{t \geq 0}$.
\begin{lemma}
\label{lem:llinbound}
Under the setting of this section and if $\|P_t (\|g_0\|^2)\|_{L^\infty(\mathcal{M})} \leq \frac{C}{1-e^{-2t}}$, we have 
\begin{multline*}
-\frac{2}{n} \E\left[\sum_{i=1}^n \left(l^{lin}_{\hat{\score}_t, t} (X_i) - \E[l^{lin}_{\hat{\score}_t, t}(X_i)]\right)\right]  \leq \\
\frac{1}{8}\E\left[\|\hat{\score}_t - g_0\|^2_{L^2(\mu_t)}\right] + \frac{C}{(1-e^{-2t})^{3/2}}\left(r^2 + \frac{1}{ n}\right),
\end{multline*}
with
\[
r^2 = \begin{cases}
\left(n^{2(s+2+\ell)} (1-e^{-2t})^{3d/2-(s+2+\ell)} \window^{2\ell d}\right)^{-\frac{1}{2(s+2+\ell)+d}} \text{ if $\sqrt{1-e^{-2t}} \leq h$} \\
   \left(n^{2(s+2+m)} (1-e^{-2t})^{d(m+3/2)-(s+2+m)} \right)^{-\frac{1}{2(s+2+m)+d}}  \text{ otherwise}
\end{cases},
\]
provided $r^2 \leq \sqrt{1-e^{-2t}}$
\end{lemma}

We then deal with the quadratic component of the loss by approximating functions of $\mathcal{F}$ by polynomial functions around $\mathcal{M}$ and conclude using simple concentration inequalities. This yields the following result.
\begin{lemma}
\label{lem:lquadbound}
In the setting of this section, we have 
\begin{multline*}
-\E\left[\frac{1}{n} \sum_{i=1}^n \left(l^{quad}_{\hat{\score}_t,t} (X_i) - \E[l^{quad}_{\hat{\score}_t,t}(X_i)]\right) \right] \leq \\
\frac{1}{8} \E\left[\|\hat{\score}_t-g_0\|_{L^2(\mu_t)}^2\right] +  \frac{C}{(1-e^{-2t})^2} \left(\frac{\log(n)}{n}\right)^{\frac{-2(s+1)}{d}}.
\end{multline*}
\end{lemma}
\begin{remark}
Up to the logarithmic terms, the error term obtained here corresponds to the minimax rate of $n^{-\frac{k}{d}}$ appearing in manifold estimation \cite{aamarilevrard2019}. This is not surprising as estimating the measure $\mu$ requires estimating the manifold itself. 
\end{remark} 

Combining the various bounds we obtained so far yields our final result. 
\begin{proposition}
\label{pro:final-variance}
Under the setting of this section and if $\|P_t (\|g_0\|^2)\|_{L^\infty(\mathcal{M})} \leq \frac{C}{1-e^{-2t}}$, we have \begin{multline*}
\E\left[(\widehat{L}_t(g_0) - \widehat{L}_t(\hat{\score}_t)) -  (L_t(g_0) - L_t(\hat{\score}_t))\right] \leq \\
\frac{1}{4} \E\left[\|\hat{\score}_t-g_0\|_{L^2(\mu_t)}^2\right] + \frac{C}{(1-e^{-2t})^2} \left(r^2 +  \left(\frac{\log(n)}{n}\right)^{\frac{-2(s+1)}{d}}\right)
\end{multline*}
with 
\[
r^2 = \begin{cases}
\left(n^{2(s+2+\ell)} (1-e^{-2t})^{d-2(s+2+\ell)} \window^{2\ell d}\right)^{-\frac{1}{2(s+2+\ell)+d}} \text{ if $\sqrt{1-e^{-2t}} \leq h$} \\
   \left(n^{2(s+2+m)} (1-e^{-2t})^{d(m+1)-2(s+2+m)} \right)^{-\frac{1}{2(s+2+m)+d}}  \text{ otherwise}
\end{cases},
\]
provided $r^2 \leq 1-e^{-2t}$.
\end{proposition}
\noindent In order to obtain Theorem~\ref{thm:maindiffusion} from this result, observe that 
\[
\|\hat{\score}_t-g_0\|^2_{L^2(\mu_t)} \leq 2 \|g_0 - \rho_t^\star\|_{L^2(\mu_t)}^2 + 2 \|\hat{\score}_t - \rho_t^\star\|_{L^2(\mu_t)}^2.
\]

\subsection{Consequences for generative algorithms}
\label{sec:consequences}
Suppose Assumption~\ref{hyp:diffusion} is verified and let $X_1, \dots, X_n$ be i.i.d. random variables drawn from $\mu$. As explained in the introduction, one way to evaluate the quality of our estimator is to use it in an SGM and compare the measure the SGM samples from with the target measure in terms of Wasserstein-$1$ distance $W_1$. Let us recall that for two probability measures $\mu$ and $\nu$, the Wasserstein-$1$ distance is defined by:
\begin{align*}    
W_1\big(\mu, \nu \big) & = \inf_{\pi\in C(\mu,\nu)} \int_{\R^D\times \R^D}\|x-y\| \pi(dx,dy) \\
&=   \sup_{ \|f\|_{\mbox{\scriptsize{Lip}}}\leq 1} \Big\{ \int_{\R^D} f \ d\mu - \int_{\R^D}  f\ d\nu\Big\},
\end{align*}
where $C(\mu,\nu)$ is the set of probability measures on $\R^D\times \R^D$ with marginal distributions $\mu$ and $\nu$, and where the second equality is the dual formulation, see \cite[Remark 6.5]{villani}.\medbreak

Let $T = \frac{s+1}{2s+d} \log(n)$.
Let us now specify the score estimator we will use. Let $\window > 0$  and $\ell > 0$ such that $2(s+2+\ell) > d$. Let $C',R$ be positive constants and consider the following score estimator, for $t>0$, 
\begin{equation}\label{eq:tilde-score}
\tilde{\score}_t = \begin{cases}
c(\hat{\score}_t) & \text{ with $\mathcal{F} = \mathcal{F}_{R}$ if $1-e^{-2t} \geq \frac{C'}{\log{n}}$} \\
c(\hat{\score}_t) & \text{ with $\mathcal{F} = \mathcal{F}_{t, s, R, \window, \ell}$ if $1-e^{-2t} < \frac{C'}{\log(n)}$}
\end{cases},
\end{equation}
where $\hat{\score}_t$ is defined in \eqref{eq:hat-score}, $\mathcal{F}_R$ in Remark \ref{rem:large-times} and $c$ is a clipping function defined by 
\[
\forall y \in \R^D, c(y) = \begin{cases}
y \text{ if $\|y\|\leq  \sqrt{\frac{4 \log(n)}{1-e^{-2t}}}$} \\
\frac{y}{\|y\|} \sqrt{\frac{4 \log(n)}{1-e^{-2t}}} \text{ otherwise}
\end{cases}.
\]
\begin{remark}
The clipping part is introduced in order to plug our estimator into Lemma D.7 \cite{oko} which turns bounds on score function estimations errors into bounds on the Wasserstein distance $W_1(\hat{\mu}^{SGM}_n, \mu)$.
\end{remark}

\begin{remark}
We mostly focused on the behavior of the score function around the manifold in our analysis, which is mostly relevant for small times. In order to obtain optimal rates for the SGM, we thus need to use different estimators for  ``large" times (that is when $1-e^{-2t} \geq \frac{C}{\log{n}}$) and ``small" times. Such a distinction is likely to be unnecessary in practice. 
\end{remark}

For this score estimator, we can consider the stochastic process $( \widehat{\overleftarrow{X}}_t )_{t \in [0,T]}$ defined by $\widehat{\overleftarrow{X}}_0 \sim \gamma$ and
\begin{equation}\label{def:diffusion-munSGM}
\forall t \in [0,T), d  \widehat{\overleftarrow{X}}_t = \big( \widehat{\overleftarrow{X}}_t + 2 \tilde{\score}_{T-t}( \widehat{\overleftarrow{X}}_t)\big) dt + \sqrt{2} dW_t,
\end{equation}
where $(\tilde{\score}_{t})_{t > 0}$ is an estimator of $(\score^\star_{t})_{t > 0}$.
The SGM associated to this stochastic process samples from the distribution of $\widehat{\overleftarrow{X}}_{T}$, denoted by $\hat{\mu}^{SGM}_n$. The following result provides a convergence rate for this measure to $\mu$.

\begin{corollary}
\label{cor:finalresult}
If $d \geq 3$, then there exist $R_0, C'_0 > 0$ such that taking $C' \leq C'_0, R \geq R_0$ and $\window = n^{\frac{-1}{2s+d}}$ in the previous setting yields  
\[
\E[W_1(\hat{\mu}^{SGM}_n, \mu)] \leq C \log(n)^{3/2} n^{-\frac{s+1}{2s+d}},
\]
for some $C >0$ independent of $n$.
\end{corollary}
\noindent As noted in the introduction, this convergence rate is optimal (up to the logarithmic factor) \cite{Divol2022}.

\section{Proofs}
\label{sec:proofs}

In our analysis, we will often rely on smoothing properties of the semigroup $(P_t)_{t \geq 0}$, summarized in the following result. 
\begin{lemma}
\label{lem:integrationByParts}
Let $t >0$. For any test function $\test$ and any $k \in \N$,
\[
\nabla^k P_t \test = e^{-kt} P_t \nabla^k \test
\]
and, for any multi-index $j \in \{1, \dots, D\}^k$,
\[
(\nabla^k P_t \test)_{j}^2 \leq \frac{k! \ e^{-2kt}}{(1-e^{-2t})^k} P_t \test^2.
\]
\end{lemma}

\begin{proof}
Let $x \in \mathbb{R}^D$. We have
\[
P_t \test (x) = \int_{\R^D} \test(e^{-t} x  + \sqrt{1-e^{-2t}}z) \gamma(z) dz,
\]
Therefore, for any integer $k \geq 1$, 
\begin{align*}
\nabla^k P_t \test (x) & =  \int_{\R^D} \nabla^k \left(\test(e^{-t} x  + \sqrt{1-e^{-2t}}z)\right) \gamma(z) dz \\
& = e^{-kt} \int \nabla^k \test(e^{-t} x  + \sqrt{1-e^{-2t}}z) \gamma(z) dz.
\end{align*}
Furthermore, for any $k \in \mathbb{N}$, successive integrations by parts with respect to $\gamma$ yield 
\begin{align*}
\nabla^k P_t \test(x) & = e^{-kt} \int_{\R^D} \nabla^k \test(e^{-t} x  + \sqrt{1-e^{-2t}}z) \gamma(z) dz \\
& = \frac{e^{-kt}}{(1-e^{-2t})^{k/2}} \int_{\R^D} H_k(z) \test(e^{-t} x  + \sqrt{1-e^{-2t}}z) \gamma(z) dz,
\end{align*}
where $H_k$ is the $k$-th Hermite polynomial tensor. Finally, using Cauchy-Schwarz's inequality yields, for any $j \in \{1, \dots, D\}^k$,
\[
(\nabla^k P_t \test)_{j}^2   \leq  \frac{e^{-2kt}}{(1-e^{-2t})^k}  \int_{\R^D} (H_k(z))_j^2 \gamma(z) dz \int_{\R^D} \test(e^{-t} x  + \sqrt{1-e^{-2t}}z)^2 \gamma(z) dz,
\]
We finally conclude the proof by remarking that $\left\| \int_{\R^D} H_k(z)^2 \gamma(z) dz \right\|_\infty \leq k!$.
\end{proof}

\subsection{Proof of Lemma~\ref{lem:bias-variance}}
We have, for any $g \in W^{1,2}(\mu_t)$,
\[
\|g - \score^\star_t\|^2_{L^2(\mu_t)}  = L_t(g) - L_t(\score^\star_t).
\]
See \cite{Vincent2011}. Applying this result for $g=\hat{\score}_t$ and $g=g_0$, we obtain
\begin{align*}
\|\hat{\score}_t - \score^\star_t\|^2_{L^2(\mu_t)} & = L_t(\hat{\score}_t) - L_t(\score^\star_t) \\
& = L_t(\hat{\score}_t) - L_t(g_0) + L_t(g_0) - L_t(\score^\star_t) \\
& = \|g_0 - \score^\star_t\|^2_{L^2(\mu_t)}-(L_t(g_0) - L_t(\hat{\score}_t)).
\end{align*}
Finally, since $\hat{\score}_t$ minimizes $\widehat{L}_t$ over $\mathcal{F}$ we also have 
\[
\widehat{L}_t(g_0) - \widehat{L}_t(\hat{\score}_t) \geq 0,
\]
so that
\[
\|\hat{\score}_t - \score^\star_t\|^2_{L^2(\mu_t)} \leq \|g_0 - \score^\star_t\|^2_{L^2(\mu_t)} + (\widehat{L}_t(g_0) - \widehat{L}_t(\hat{\score}_t)) -  (L_t(g_0) - L_t(\hat{\score}_t)). 
\]

\subsection{Proof of Lemma~\ref{lem:parallel-regularity}}

Let $t < 1, x_0 \in \mathcal{M}_t^{\tau/2}$. In the following, we denote by $C$ a generic constant which is independent of $t$ and of $x_0$. We proceed in two steps: we first bound $f^\|_t(x_0)$ from below before bounding $\|\nabla^i f_t^\parallel(x_0)\|$ from above for all $1 \leq i \leq s+2$. The upper-bound on the derivatives of the score function will then follow. Let $0 < \tau < \tau_{\min}$ be sufficiently small and $x \in \mathcal{M}_t^{\tau/2}, y \in \mathcal{M}_t$ and  
\[
K_t(x,y) = e^{-\frac{\|p_t(x)-y\|^2 + 2 \left<x - p_t(x) ,p_t(x) - y \right>}{2(1-e^{-2t})}}.
\]
Let $B_r = \{y \in \mathcal{M}_t \mid \|y - p_t(x_0)\| \leq e^{-t} r\}$, where $r$ is defined in our assumptions on $\mathcal{M}$, and let $\bar{B_r} = \mathcal{M}_t \setminus B_r$. We have
\[
f^\|_t(x) = \frac{1}{\left(2\pi (1-e^{-2t})\right)^{d/2}} \left(\int_{B_r}+\int_{\bar{B}_r} \right) K_t(x,y) d\nu_t(y).
\]
By Assumption~\ref{hyp:diffusion}, the orthogonal projection $\pi:B_r \rightarrow T_{p_t(x_0)} \mathcal{M}_t$ admits a $k$-times differentiable inverse $\Psi: \pi(B_r) \rightarrow B_r$ with bounded derivatives. Therefore, a change of variable $y = \Psi(z)$ gives
\[
\int_{B_r} K_t(x,y) d\nu_t(y) = \int_{\pi(B_r)} K_t(x,\Psi(z)) f^{\nu_t}(\Psi(z)) \sqrt{\det(D\Psi(z)^T D\Psi(z))} dz,
\]
where $D\Psi(z)$ is the Jacobian of $\Psi$ and $f^{\nu_t}$ is the density of $\nu_t$ with respect to the volume measure of $\mathcal{M}_t$. Remark that, as $\mathcal{M}_t$ is a $d$-dimensional submanifold of $\R^D$, we can identify $T_{p_t(x_0)} \mathcal{M}_t$ with $\R^d$ and, since $\Psi$ is the inverse of the orthogonal projection $\pi$ on  $T_{p_t(x_0)} \mathcal{M}_t$, there exists $\psi$ such that 
\[
\forall z \in \pi(B_r), \Psi(z) = (z, \psi(z))
\]
so that
\begin{multline*}
\int_{B_r} K_t(x,y) d\nu_t(y) = \\
\int_{\pi(B_r)} K_t(x,(z, \psi(z))) f^{\nu_t}(z,\psi(z)) \sqrt{det(I_d+D \psi(z)^T D \psi(z))} dz.
\end{multline*}
In particular, letting $\tilde{f}(z) = f^{\nu_t}(z,\psi(z)) \sqrt{det(I_d+D \psi(z)^T D \psi(z))}$,
\begin{equation}
\label{eq:changeofvariable}
\int_{B_r} K_t(x,y) d\nu_t(y) = \int_{\pi(B_r)} K_t(x,(z, \psi(z))) \tilde{f}(z) dz.
\end{equation}

\noindent \textbf{Step 1:} Let us first show $f^\|_t$ is bounded from below. By Lemma~4.8 \cite{Federer59}, $x_0-p_t(x_0)$ is orthogonal to $T_{p_t(x_0)} \mathcal{M}_t$. Hence, by Theorem~4.18 \cite{Federer59} and since $x_0 \in \mathcal{M}_t^{\tau/2}$, 
\[
\left|\left<x_0-p_t(x_0), p_t(x_0) -y\right>\right| \leq \|x_0-p_t(x_0)\| \frac{d_{\mathcal{M}_t}(p_t(x_0), y)^2}{2 e^{-t} \tau_{\min}} \leq \frac{d_{\mathcal{M}_t}(p_t(x_0), y)^2}{4},
\]
where $d_{\mathcal{M}_t}$ is the geodesic distance on the manifold $\mathcal{M}_t$. Since this manifold is compact with a non-zero reach, we have, for $y \in B_r$,
\[
d_{\mathcal{M}_t}(p_t(x_0), y)^2 \leq 2 \sqrt{C} \|p_t(x_0)-y\|^2
\]
so that 
\[
\left|\left<x_0-p_t(x_0), p_t(x_0) -y\right>\right| \leq C \|p_t(x_0) - y\|^2.
\]
Therefore,
\begin{equation}
\label{eq:InBrBound}
K_t(x_0,y) \geq e^{-\frac{C\|p_t(x_0)-y\|^2}{1-e^{-2t}}}
\end{equation}
or, by decomposing $p_t(x_0)$ and $y$ on $T_{p_t(x_0)} \mathcal{M}_t$ and its orthogonal complement,
\[
K_t(x_0, \Psi(z)) \geq e^{-\frac{C}{1-e^{-2t}}\left(\|\pi(p_t(x_0))-z\|^2 + \|\psi(\pi(p_t(x_0))) - \psi(z)\|^2\right)}
\]
so that 
\begin{multline*}
f^\|_t(x_0) \geq  \\
\frac{\min_{w \in \pi(B_r)}\left(\tilde{f}(w)\right)} {\left(2\pi (1-e^{-2t})\right)^{d/2}}  \int_{\pi(B_r)} e^{-\frac{C}{1-e^{-2t}}\left(\|\pi(p_t(x_0))-z\|^2 + \|\psi(\pi(p_t(x_0))) - \psi(z)\|^2\right)}  dz.
\end{multline*}
Using Theorem~4.18 \cite{Federer59} once more, we obtain 
\[
\|\psi(\pi(p_t(x_0))) - \psi(z)\| \leq C \|\pi(p_t(x_0))-z \|^2
\]
so that 
\[
f^\|_t(x_0) \geq \frac{1}{\left(2\pi (1-e^{-2t})\right)^{d/2}}  \int_{\pi(B_r)} e^{-\frac{C\|\pi(p_t(x_0))-z\|^2}{2(1-e^{-2t})} \left(1+ C \|\pi(p_t(x_0))-z\|^2\right)} \min\left(\tilde{f}\right)  dz.
\]
Finally, by Proposition 8.7 \cite{aamarilevrard2018}, $B\left(p_t(x_0), \frac{7e^{-t}r}{8}\right) \subset \pi(B_r)$ thus translating by $\pi(p_t(x_0))$ yields
\[
f^\|_t(x_0) \geq \frac{1}{\left(2\pi (1-e^{-2t})\right)^{d/2}}  \int_{\|z\| \leq \frac{7e^{-t}r}{8}} e^{-\frac{C\|z\|^2}{1-e^{-2t}}\left(1 + C \|z\|^2\right)} \min\left(\tilde{f}\right)  dz.
\]
When $t$ converges to $0$, the right-hand term converges to a non-zero constant. Therefore, 
\begin{equation}
\label{eq:lowerbounddensity}
f^\|_t(x_0) \geq C > 0.
\end{equation}

\noindent \textbf{Step 2:} Let us now show that $\|\nabla^i f^\|_t(x_0)\|$ is bounded from above. As before, we have 
\[
\nabla^i f^\|_t(x) =  \frac{1}{\left(2\pi (1-e^{-2t})\right)^{d/2}} \left(\int_{B_r}+\int_{\bar{B}_r} \right) \nabla^i_x K_t(x,y) d\nu_t(y).
\]
Let us start by dealing with the integral over $\bar{B}_r$. Since $p_t$ is $k-1$ times differentiable and $i \leq k-1$,
\[
\|\nabla^{i}_x K_t(x_0,y)\| \leq \frac{C}{(1-e^{-2t})^i} K_t(x_0,y).
\]
Then, in the same way we derived Equation~(\ref{eq:InBrBound}), there exist $C_1, C_2 >0$ independent of $t$ such that, when $y \in \bar{B}_r$ and $x$ is sufficiently close to $x_0$,
\[
K_t(x,y) \leq e^{-\frac{C_1r^2 - C_2 \tau r}{1 - e^{-2t}}}
\]
and thus, since $\tau$ can be taken small enough so that $C_1r^2 - C_2 \tau r > C > 0$, we have
\begin{equation}
\label{eq:barBr}
\frac{1}{\left(2\pi (1-e^{-2t})\right)^{d/2}} \int_{\bar{B}_r} \left\| \nabla^i_x K_t(x_0,y) \right\| d\nu_t(y) \leq \frac{C}{\left(2\pi (1-e^{-2t})\right)^{d/2+i}} e^{-\frac{C}{1-e^{-2t}}}
\end{equation}
which is uniformly bounded for all $t < 1$. We are thus left with dealing with the integral part over $B_r$. Recall that Equation~\eqref{eq:changeofvariable} gives
\[
\int_{B_r} K_t(x,y) d\nu_t(y) = \int_{\pi(B_r)} K_t(x,(z, \psi(z))) \tilde{f}(z) dz.
\]
Let $x \in \R^D$ be sufficiently close to $x_0$ so that $p_t(x)$ belongs to $B_r$. Letting $w = \pi(p_t(x))$, we have 
\[
K_t(x,(z,\psi(z))) = 
e^{-\frac{\|w-z\|^2 + \|\psi(w)-\psi(z)\|^2 + 2 \left<x - p_t(x), (w-z, \psi(w) - \psi(z)) \right>}{2(1-e^{-2t})}}.
\]
Using Taylor's theorem then yields 
\[
\psi(w) - \psi(z) = \nabla \psi(w) (w - z) 
- \Phi(w,z) (w - z)^{\otimes 2},
\]
where $(w - z)^{\otimes 2}$ is the matrix $(w - z) (w - z)^T$ and 
\[
\Phi : (w,z) \rightarrow \left(\int_{0}^1 (1-u) \nabla^2 \psi\big(w + u(z-w)\big) du \right),
\]
Remark that, by definition of $\psi$, $(w-z, \nabla\psi(w)(w-z)) \in T_{p_t(x)}\mathcal{M}_t$. Thus, as $x - p_t(x)$ is orthogonal to $T_{p_t(x)}\mathcal{M}_t$, 
\[
\left<x - p_t(x), (w - z, \nabla \psi(w) (w - z)) \right> = 0
\]
and
\[
2 \left<x - p_t(x) ,(w - z, \psi(w) - \psi(z)) \right> = \\
-2\left<(x - p_t(x))^\perp, \Phi(w,z) (w - z)^{\otimes 2} \right>,
\]
where $(x-p_t(x))^\perp$ denotes the projection of $x-p_t(x)$ on the orthogonal complement of $T_{p_t(x_0)} \mathcal{M}_t$. Hence,
\[
K_t(x,(z,\psi(z))) = 
e^{-\frac{\|w-z\|^2 + \|\psi(w)-\psi(z)\|^2 -2\left<(x - p_t(x))^\perp, \Phi(w,z) (w - z)^{\otimes 2} \right>}{2(1-e^{-2t})}}.
\]
Let $(e_1, \dots, e_d)$ be an orthonormal basis of $T_{p_t(x_0)}\mathcal{M}_t$ and $(e_{d+1}, \dots, e_D)$ be such that $(e_1, \dots, e_D)$ is an orthonormal basis of $\R^D$. Let $x^\perp$ be the projection of $x$ on the linear subspace spanned by $(e_{d+1}, \dots, e_D)$. For $x$ sufficiently close to $x_0$, the map $x \rightarrow (\pi \circ p_t(x), x^\perp)$ is $(k-1)$-times differentiable with invertible Jacobian \cite{Leobacher2021}. Thus, around $x_0$, we can use the inverse function theorem to find a $(k-1)$-times differentiable function $u$ such that 
\[
x = u(w, x^\perp).
\]
Since this function $u$ is $(k-1)$-times differentiable, it is sufficient to show that the derivatives of 
\begin{multline*}
\left(2\pi (1-e^{-2t})\right)^{d/2} J_t(w, x^\perp) = \\
\int_{\pi(B_r)} e^{-\frac{\|w-z\|^2 + \|\psi(w)-\psi(z)\|^2 -2\left<(x - p_t(x))^\perp, \Phi(w,z) (w - z)^{\otimes 2} \right>}{2(1-e^{-2t})}} \tilde{f}(z) dz.
\end{multline*}
with respect to $w$ and to $x^\perp$ are properly bounded with respect to $t$ to conclude the proof. Let $i \leq s+2, j \in\{0,1,2\},j^w \leq s$ and $j^\perp$ such that $j + j^w + j^\perp = i$. If we can show that the quantities $\left\|\nabla^j \nabla_w^{j^w} \nabla_{x^\perp}^{j^\perp}  J_t(w, x^\perp)\right\|$ are bounded for all $(j,j^w, j^\perp)$, then we will have shown that $\|\nabla^i J_t(w,x^\perp)\|$ are bounded as well. First, letting
\[
\Phi^{(0)} = \left(\frac{\Phi(w,z) (w - z)^{\otimes 2}}{1-e^{-2t}}\right)^{\otimes j^\perp},
\]
we have 
\begin{multline*}
\left(2\pi (1-e^{-2t})\right)^{d/2} \nabla_{x^\perp}^{j^\perp} J_t(w, x^\perp) = \\
 \int_{\pi(B_r)} e^{-\frac{\|w-z\|^2 + \|\psi(w)-\psi(z)\|^2 - 2\left<(x^\perp - \psi(w)) \Phi(w,z), (w - z)^{\otimes 2} \right>}{2(1-e^{-2t})}} \Phi^{(0)}(w,z) \tilde{f}(z) dz.
\end{multline*}
Now, letting $\Psi_w(w,z) = -2 (x^\perp - \psi(w)) \Phi(w,z)$, we have 
\begin{multline*}
\left(2\pi (1-e^{-2t})\right)^{d/2} \nabla_{x^\perp}^{j^\perp} J_t(w, x^\perp) = \\
\int_{\pi(B_r)} e^{-\frac{\left<I_d + \Psi_w(w,z), (w-z)^{\otimes 2}\right> + \|\psi(w)-\psi(z)\|^2}{2(1-e^{-2t})}} \Phi^{(0)}(w,z) \tilde{f}(z) dz.
\end{multline*}
Letting 
\[
K_t'(w,z) = \frac{1}{\left(2\pi (1-e^{-2t})\right)^{d/2}}  e^{-\frac{\left<I_d + \Psi_w(w,z), (w-z)^{\otimes 2}\right> + \|\psi(w)-\psi(z)\|^2}{2(1-e^{-2t})}},
\]
we have 
\begin{multline*}
\nabla_w K_t'(w,z)  = -\nabla_z K_t'(w,z) 
- K_t'(w,z)\left[\left(\nabla_w + \nabla_z \right)\Psi_w(w,z)\right] \frac{(w-z)^{\otimes 2}}{2(1-e^{-2t})} \\
- K_t'(w,z) \left(\nabla_w + \nabla_z \right) \frac{\|\psi(w)-\psi(z)\|^2}{2(1-e^{-2t})}
\end{multline*}
or
\begin{multline*}
\nabla_w K_t'(w,z)  = -\nabla_z K_t'(w,z) 
- K_t'(w,z)\left[\left(\nabla_w + \nabla_z \right)\Psi_w(w,z)\right] \frac{(w-z)^{\otimes 2}}{2(1-e^{-2t})} \\
- K_t'(w,z) \frac{(\nabla \psi(w) - \nabla \psi(z)) (\psi(w)-\psi(z))}{1-e^{-2t}}.
\end{multline*}
Similarly,
\[
\nabla_w \Phi^{(0)}(w,z) = - \nabla_z \Phi^{(0)}(w,z) + J,
\]
with
\begin{multline*}
J =  \sum_{l = 1}^{j^\perp} \left(\frac{\Phi(w,z) (w - z)^{\otimes 2}}{1-e^{-2t}}\right)^{\otimes l-1} \otimes  \left( \frac{\left[(\nabla_w + \nabla_z) \Phi(w,z)\right] (w-z)^{\otimes 2}}{(1-e^{-2t})} \right)  \\
\otimes \left(\frac{\Phi(w,z) (w - z)^{\otimes 2}}{1-e^{-2t}}\right)^{\otimes j^\perp - l}.
\end{multline*}
Therefore, letting $I = K_t'(w,z) \Phi^{(0)}(w,z)$ and applying Green's theorem  yields 
\begin{multline*}
\nabla_w \nabla_{x^\perp}^{j^\perp} J_t(w, x^\perp) = \int_{\pi(B_r)} I \nabla_z \tilde{f}(z) dz - \int_{\partial \pi(B_r)} I \tilde{f}(z) n(z) dS(z) \\
- \int_{\pi(B_r)} \frac{(\nabla \psi(w) - \nabla \psi(z)) (\psi(w)-\psi(z))}{1-e^{-2t}} \otimes I  \tilde{f}(z) dz \\
- \int_{\pi(B_r)}  \left[\left(\nabla_w + \nabla_z \right)\Psi_w(w,z) \right] \frac{(w-z)^{\otimes 2}}{2(1-e^{-2t})} \otimes I \tilde{f}(z) dz +  \int_{\pi(B_r)}  J K'_t(w,z) \tilde{f}(z) dz,
\end{multline*}
where $dS(z)$ and $n(z)$ are respectively the surface measure and the outward normal of $\partial \pi(B_r)$. Let us note that, similarly to how we derived Equation~(\ref{eq:barBr}), the derivatives of 
\[
\int_{\partial \pi(B_r)} I \tilde{f}(z) n(z) dS(z) = 
\int_{\partial \pi(B_r)} K_t'(w,z) \Phi^{(0)}(w,z) \tilde{f}(z) n(z) dS(z)
\]
are bounded independently of $t < 1$. Therefore, by successive differentiations and applications of Green's theorem, we obtain 
\[
\nabla_w^{j^w}\nabla_{x^\perp}^{j^\perp}  J_t(w, x^\perp) = \Phi_t^{'(j^w)}(w) + \int_{\pi(B_r)} K_t'(w,z) \Phi^{(j^w)}(w,z) dz,
\]
where $\|\Phi_t^{'(j^w)}\|$ is bounded independently of $t < 1$ and 
\[
\Phi^{(j^w)}(w,z) = \sum_{l = 1}^{K} \xi^{(l)} \bigotimes_{m = 1}^{k_l} \frac{\phi^{(m,l)}(w) - \phi^{(m,l)}(z)}{\sqrt{1-e^{-2t}}},
\]
where $K, k_1, \dots, k_K$ are nonnegative integers, the $ \xi^{(l)}$ are bounded and the $\phi^{(m,l)}$ have bounded first derivatives. 
From here, using Taylor expansions, we obtain that
\[
\|\Phi^{(j_w)}(w,z)\| \leq C \sum_{l = 1}^{K} \frac{\|w-z\|^{k_l}}{(1-e^{-2t})^{k_l/2}}
\]
and thus
\begin{multline*}
\left\|\int_{\pi(B_r)}  \nabla^j K_t'(w,z) \Phi^{(j_w)}(w,z) dz \right\| \leq \\
C \int_{\R^d}  \sum_{l = 1}^{K} \frac{\|w-z\|^{k_l}}{(1-e^{-2t})^{k_l/2}} \left\| \nabla^j K'_t(w,z) \right\|  dz.
\end{multline*}
If $j=0$, after bounding $K'_t(w,z)$ following the arguments used to derive Equation~\eqref{eq:barBr}, we obtain that the previous quantity is bounded independently of $t$. Therefore, we obtained 
\[
\forall 1 \leq i \leq s, \sup_{x \in \mathcal{M}^{\tau/2}_t} \|\nabla^i f^\|_t(x)\| \leq C.
\]
Finally, remarking that differentiating either $K'_t$ or $\Phi^{(j_w)}$ once more yields a $\frac{1}{\sqrt{1-e^{-2t}}}$ factor while the derivative of $\Phi'^{(j_w)}$ stays bounded, we have 
\[
\sup_{x \in \mathcal{M}^{\tau/2}_t} \|\nabla^{s+1} f^\|_t(x)\|\leq \frac{C}{\sqrt{1-e^{-2t}}}.
\]
Similarly,
\[
\sup_{x \in \mathcal{M}^{\tau/2}_t} \|\nabla^{s+2} f^\|_t(x)\| \leq \frac{C}{1-e^{-2t}}.
\]
We can then conclude the proof by combining these bounds with Equation~\eqref{eq:lowerbounddensity}.

\subsection{Proof of Lemma~\ref{lem:cutscore}}
We define $\tilde{\score}_t$ by clipping the original score function. Let $\tau$ be the constant defined in Lemma~\ref{lem:parallel-regularity} and $c : \R^+ \rightarrow \R$ be a $C^\infty$ cutoff function such that $0 \leq c \leq 1, c(x) = 1$ for $x < \frac{1}{\sqrt{2}}$ and $c(x) = 0$ for $x > 1$ and let 
\[
\tilde{\score}_t : x \rightarrow \score_t^\star(x) c\left(\left(\frac{2e^t d(x, \mathcal{M}_t)}{\tau} \right)^2 \right).
\]
Remark that $c$ is $C^\infty$ and that for all $x \in \R^D$ such that $c\left(\left(\frac{2e^t d(x, \mathcal{M}_t)}{\tau} \right)^2\right) > 0$, we have that $x \rightarrow d(x, \mathcal{M}_t)^2 = \|x - p_t(x)\|^2$ is well-defined and $k$-times differentiable. Therefore, the first claim follows from Lemmas~\ref{lem:density-decomposition} and \ref{lem:parallel-regularity}.

Finally, by Equation~(\ref{eq:conditionalscoredef}),
\begin{align*}
    \|\tilde{\score}_t - \score_t^\star\|_{L^2(\mu_t)}^2 & \leq \int_{\R^D}  \|\score_t^\star(x)\|^2 1_{4 d(x, \mathcal{M}_t) \geq e^{-t} \tau} d \mu_t \\
     &\leq \frac{\E[\|\E[Z_t \mid \overrightarrow{X}_t] \|^2 1_{4 d(\overrightarrow{X}_t, \mathcal{M}_t) \geq e^{-t} \tau} ] }{1-e^{-2t}} \\
    & \leq \frac{\E[\|Z_t\|^2 1_{4 d(\overrightarrow{X}_t, \mathcal{M}_t) \geq e^{-t} \tau} ] }{1-e^{-2t}} \\
    & \leq \frac{\E[\|Z_t\|^2 1_{4 \|\overrightarrow{X}_t - e^{-t}X\| \geq e^{-t} \tau} ] }{1-e^{-2t}} \\
    & \leq \frac{\E[\|Z_t\|^2 1_{4 \sqrt{1-e^{-2t}} \|Z_t\| \geq e^{-t} \tau} ] }{1-e^{-2t}}.
\end{align*}
The second claim of the result thus follows from tail properties of the Gaussian measure.

\subsection{Proof of Lemma~\ref{lem:llinbound}}
Let $g \in \mathcal{F}$ and set
\[
\tilde{l}^{lin}_{g,t} = C(1-e^{-2t})^{3/2} l^{lin}_{g,t}
\]
so that, by Equation~(\ref{eq:Llinbound}),
\[
|\tilde{l}^{lin}_{g,t}| \leq 1.
\]
Remark that we also have 
\[
\E[\tilde{l}^{lin}_{g, t}(X_1)^2] \leq \sigma(g,t)^2 :=  C (1-e^{-2t})^2 \|g - g_0\|^2_{L^2(\mu_t)}.
\]
Now, let $r > 0, \mathcal{F}_r =\left\{ g \in \mathcal{F} \mid \sigma(g,t) \leq  r \right\}$ and 
\[
S = \frac{1}{\sqrt{n}} \sup_{g\in \mathcal{F}_r} \left|\sum_{i=1}^n  \tilde{l}^{lin}_{g,t}(X_i) - \E[\tilde{l}^{lin}_{g,t}(X_i)]\right|.
\]
Let $g \in \mathcal{F}_r, x \in \mathcal{M}$ and $2 \leq i \leq s+2+\ell$. By our assumption on $g_0$, Lemma~\ref{lem:integrationByParts} and the general Leibniz rule, we have 
\begin{align*}
\|\nabla^i l_{g,t}^{lin}(x)\| & \leq \left\|\nabla^2 P_t\left(\nabla^{i-2} (\div(g-g_0)) \right)(x)\right\| + \left\|\nabla P_t \left(\nabla^{i-1} \left<g - g_0, g_0\right>\right)(x) \right\|\\
& \leq \frac{C}{1-e^{-2t}} \sqrt{P_t\left(\left\|\nabla^{i-2} (\div(g-g_0)) \right\|^2 \right)(x)} \\
& \hspace{2cm} + \left\|\nabla P_t \left(\sum_{j= 0}^{i-1} \binom{i-1}{j} \nabla^j(g-g_0) \nabla^{i-1-j} g_0 \right)(x) \right\| \\
& \leq \frac{C}{h^\ell (1-e^{-2t})^2} + \left\|P_t \left(\nabla \sum_{j= 1}^{i-2} \binom{i-1}{j} \nabla^j(g-g_0) \nabla^{i-1-j} g_0 \right) (x) \right\| & \\
& \hspace{2cm} + \left\|\nabla P_t  \left(\nabla^{i-1}(g-g_0) g_0)\right)(x) \right\|  + \left\|\nabla P_t  \left((g-g_0) \nabla^{i-1} g_0\right)(x)\right\| \\
& \leq \frac{C}{h^\ell (1-e^{-2t})^2} + \frac{C}{\sqrt{1-e^{-2t}}}  \sqrt{P_t  \left(\left\|\nabla^{i-1}(g-g_0) g_0 \right\|^2\right)(x)} \\
& \hspace{2cm} + \frac{C}{\sqrt{1-e^{-2t}}}  \sqrt{P_t  \left(\left\|(g-g_0) \nabla^{i-1} g_0 \right\|^2\right)(x)} \\
& \leq \frac{C}{h^\ell (1-e^{-2t})^2} + \frac{C \sqrt{P_t\left(\|g - g_0\|^2\right)(x)}}{h^\ell (1-e^{-2t})^{3/2}}
\end{align*}
so that 
\[
\|l_{g,t}^{lin}\|_{W^{s+2+\ell, 2}(\mu)} \leq \frac{C}{h^\ell (1-e^{-2t})^2}\left(1+\frac{r}{\sqrt{1-e^{-2t}}}\right)
\]
or
\[
\|\tilde{l}_{g,t}^{lin}\|_{W^{s+2+\ell, 2}(\mu)} \leq \frac{C}{h^\ell \sqrt{1-e^{-2t}}}\left(1+\frac{r}{\sqrt{1-e^{-2t}}}\right).
\]
Following the arguments used in Section~\ref{sec:variancesansdiffusion} and using the entropy number bound of \cite[Theorem 2]{edmunds}, we can show that 
\[
\E[S] \leq C r  \left(\frac{1}{r h^\ell \sqrt{1-e^{-2t}}} \left(1+\frac{r}{\sqrt{1-e^{-2t}}}\right)\right)^{\frac{d}{2(s+2+\ell)}}.
\]
Remark that, as long as $r \leq \sqrt{1-e^{-2t}}$, we have 
\[
\left(\frac{1}{r h^\ell \sqrt{1-e^{-2t}}} \left(1+\frac{r}{\sqrt{1-e^{-2t}}}\right)\right)^{\frac{d}{2(s+2+\ell)}}  \leq  \left(\frac{2}{r h^\ell \sqrt{1-e^{-2t}}}\right)^{\frac{d}{2(s+2+\ell)}}.
\]

We can then use \cite[Theorem 13.19]{Massart} with, following the Theorem notation,
\begin{itemize}
    \item $\sigma(g,t)^2 =  C (1-e^{-2t})^2 \|g - g_0\|^2_{L^2(\mu_t)}$;
    \item $\rho : u \rightarrow (1-e^{-2t})^{1/4} u$;
    \item $\psi : r \rightarrow C r D_r$; 
    \item $\epsilon = C$;
\end{itemize}
to obtain
\[
-\frac{2}{n} \E\left[\sum_{i=1}^n \left(l^{lin}_{\hat{\score}_t, t} (X_i) - \E[l^{lin}_{\hat{\score}_t, t}(X_i)]\right)\right]  \leq \\
\frac{1}{8}\E\left[\|\hat{\score}_t - g_0\|^2_{L^2(\mu_t)}\right] + \frac{C}{(1-e^{-2t})^{3/2}}\left(r_\star^2 + \frac{1}{ n}\right),
\]
where $r_\star$ is the solution of $\sqrt{n} r^2 = C(1-e^{-2t})^{1/4} r   D_{(1-e^{-2t})^{1/4} r }$. Solving for $r$, we obtain 
\[
r_\star^2 \leq  C  \left(h^{2d\ell} (1-e^{-2t})^{3d/2 - (s+2+\ell)} n^{2(s+2+\ell)}\right)^{\frac{-1}{2(s+2+\ell) + d}}
\]
whenever this quantity is smaller than $\sqrt{1-e^{-2t}}$.

Finally, by Lemma~\ref{lem:integrationByParts}, we also have 
\[
\|\tilde{l}_{g,t}^{lin}\|_{W^{s+2+m, 2}(\mu)} \leq \frac{C}{\sqrt{1-e^{-2t}}^{1+m}}\left(1+\frac{r}{\sqrt{1-e^{-2t}}}\right)
\]
for any $m \in \N$. Then, as long as $2(s+2+m) > d$, we can follow the previous arguments to obtain
\[
r_\star^2 \leq \left((1-e^{-2t})^{d(m + 3/2) - (s+2+m)} n^{2(s+2+m)}\right)^{\frac{-1}{2(s+2+m) + d}},
\]
provided this quantity is smaller than $\sqrt{1-e^{-2t}}$.

\subsection{Proof of Lemma~\ref{lem:lquadbound}} 
Let $\delta > 0$. Since $\mathcal{M}$ is $d$-dimensional, we can find a maximal $\delta$-packing $\{z_1, \dots, z_K\}$ of $\mathcal{M}$ with $K \leq C \delta^{-d}$. Remark that the set of balls on $\mathcal{M}$ of radius $2\delta$ centered at $z_1, \dots, z_k$ is a covering of $\mathcal{M}$. Let $E_1, \dots E_K$ denote the Voronoi cells of $\mathcal{M}$ associated with $z_1, \dots, z_K$. If $\delta$ is sufficiently small then, by Assumption~\ref{hyp:diffusion} and Lemma A.2 \cite{aamarilevrard2019}, there exist $C_0, C_1 > 0$ such that for any of these cells $E$
\[
B(0,C_0 \delta) \subset \pi(E) \subset B(0,C_1 \delta),
\]
where $B(r)$ is the ball of radius $r$ in the tangent space of $\mathcal{M}$ at $z$.

Let $g$ be a function such that $\|g\|_{W^{s+1,\infty}(\R^D)} \leq \frac{C}{1-e^{-2t}}$ and consider one element of this packing $z$ with its associated cell $E$. On the tangent space of $\mathcal{M}$ at $z$, which we identify with $\R^d$ and center so that $\Psi_z(0) = z$, we consider the polynomial function $P_g$ of degree $s$ such that, at $0$, its first $s$ derivatives match those of $g \circ \Psi_z$. Then, by a Taylor expansion, we have 
\begin{equation}
\label{eq:boundball}
\sup_{x \in E} |g(x) - P_g(\pi_z(x))| \leq   \frac{C \delta^{s+1}}{1-e^{-2t}}.
\end{equation}

Since a polynomial of degree $s$ in dimension $d$ is described by $\binom{d+s}{d}$ coefficients, the set of polynomial functions on the cell $E$ has dimension $K$ independent of $\delta$. Let $e_1, \dots, e_K$ be an orthonormal basis (with respect to the $L_2(\mu)$ norm) of this space and, for $x \in E$, let 
\[
\Phi(x) =  \begin{cases}
(e_1(x), \dots, e_K(x)) \text{ if $x \in E$}\\
0 \text{ otherwise}
\end{cases}.
\]
Then, consider the matrix 
\[
M_i = \left(\Phi(X_i) \Phi(X_i)^T - I_K\right)
\]
so that, for any unit function of this space $P = \sum_{j=1}^K \alpha_j e_j = \alpha^T \Phi$,
\begin{align}
\label{eq:squaretomatrix}
\frac{1}{n} \left|\sum_{i=1}^n \left(1_{X_i \in E} P(X_i)^2 - \|P\|_{L_2(\mu_E)}^2\right)\right| &=  \left\| \alpha^T \left(\frac{1}{n} M_i\right) \alpha\right\| \nonumber \\
&  \leq \|P\|^2_{L_2(\mu_E)} \left\|\frac{1}{n} M_i\right\|_{op} ,
\end{align}
where $\mu_E$ is the restriction of $\mu$ to $E$.
Let us now consider one polynomial functions $P$ on $\pi_z(E)$. By changes of coordinates and using the fact that norms of polynomial functions on $\R^D$ are equivalent, we have 
\begin{align*}
\sup_{w \in \pi_z(E)} P(w)^2 & \leq \sup_{w \in B(C_1 \delta)} P(w)^2 \\
& \leq C  \sup_{w \in B(C_1)} P(\delta w)^2 \\
& \leq C \int_{B(C_0)} P(\delta w)^2 \, dw \\
& \leq C \delta^{-d} \int_{B(C_0 \delta)} P(w)^2 \, dw \\
& \leq C \delta^{-d} \int_{\pi_z(E)} P(w)^2 \, dw.
\end{align*}
Thus, since the manifold $\mathcal{M}$ is $k$-smooth and the density of $\mu$ is bounded from below, 
\[
\sup_{x \in E} P(\pi_z(x))^2 \leq C \delta^{-d} \int_E P(\pi_z(x))^2 \, d\mu.
\]
Therefore, we have 
\[
\forall x \in E, \|\Phi(x)\|^2 \leq C K \delta^{-d} \leq C \delta^{-d}
\]
so that, for any $i \in \{1, \dots, n\}$,
\[
\left\|M_i\right\|_{op} \leq  C \delta^{-d} + 1 \leq C \delta^{-d}.
\]
On the other hand, since $M_i$ is symmetric, 
\begin{align*}
\E[M_i^T M_i] & = \E[M_i^2] \\
& = \E[\Phi(X_i) \Phi(X_i)^T \Phi(X_i) \Phi(X_i)^T] - I_K \\
& = \E[\|\Phi(X_i)\|^2 \Phi(X_i) \Phi(X_i)^T] - I_k.
\end{align*}
and thus
\begin{align*}
\left\|\E[M_i^T M_i]\right\|_{op} & \leq \|\Phi(X_i)\|^2  \left\|\E[\Phi(X_i) \Phi(X_i)^T]\right\|_{op} + 1 \\
& \leq \|\Phi(X_i)\|^2 + 1 \\
& \leq C \delta^{-d}.
\end{align*}
Therefore, we can apply Corollary 6.1.2 \cite{Tropp2015} to obtain, with probability $1-C e^{-C n \delta^{-d}}$
\[
\left\|\frac{1}{n} M_{op} \right\| \leq \frac{1}{32}
\]
and, by Equation~(\ref{eq:squaretomatrix}), for all polynomial functions $P$,
\[
\frac{1}{n} \left|\sum_{i=1}^n \left(1_{X_i \in E} P(\pi_z(X_i))^2 - \|P \circ \pi_z \|_{L_2(\mu_E)}^2\right)\right| \leq \frac{1}{32} \|P \circ \pi_z\|_{L_2(\mu_E)}^2.
\]
Then, by applying a union-bound on all cells $E$ and using Equation~(\ref{eq:boundball}), we obtain that, with probability $1-C \delta^{-d} e^{-C n  \delta^d}$, for all $g$ such that $\|g\|_{W^{s+1,\infty}(\R^D)} \leq \frac{C}{1-e^{-2t}}$,
\[
 \frac{1}{n} \left|\sum_{i=1}^n \left(g(X_i)^2 - \|g\|_{L^2(\mu)}^2\right) \right| \leq \frac{1}{8} \|g\|_{L^2(\mu)}^2 + C \frac{\delta^{2(s+1)}}{(1-e^{-2t})^2}.
\]
In particular, applying this to all functions of the form $g_z : x \rightarrow (g-g_0)(e^{-t} x + \sqrt{1-e^{-2t}} z)$ for $g \in \mathcal{F}$ and integrating $z$ against the Gaussian measure, we obtain, for all $g \in \mathcal{F}$,
\[
 \frac{1}{n} \left|\sum_{i=1}^n \left(l^{quad}_{g,t} (X_i) - \E[l^{quad}_{g,t}(X_i)]\right) \right| \leq \frac{1}{8} \|g-g_0\|_{L^2(\mu_t)}^2 +  C \frac{\delta^{2(s+1)}}{(1-e^{-2t})^2}.
\]
In particular, applying the previous inequality to $\hat{\score}_t$ and taking $\delta = C  \left(\frac{\log(n)}{n}\right)^{-1/d}$, we obtain
\begin{multline*}
\frac{1}{n} \left|\sum_{i=1}^n \left(l^{quad}_{\hat{\score}_t,t} (X_i) - \E[l^{quad}_{\hat{\score}_t,t}(X_i)]\right) \right| \leq \\
\frac{1}{8} \|\hat{\score}_t-g_0\|_{L^2(\mu_t)}^2 +  \frac{C}{(1-e^{-2t})^2}  \left(\frac{\log(n)}{n}\right)^{\frac{-2(s+1)}{d}},
\end{multline*}
with probability greater than $1-C n^{\frac{-2(s+1)}{d}}$. Finally, taking the expectation concludes the proof. 

\subsection{Proof of Corollary~\ref{cor:finalresult}}
In this proof, we denote by $C$ a generic positive constant depending only on $\mu$ and $D$. Let $(t_k)_{k \in \{0, \dots, N\}}$ be a sequence such that $ \frac{1}{2} n^{-\frac{2 (s+1)}{2s +d}} \leq t_0 \leq n^{-\frac{2 (s+1)}{2s +d}}$, 
\[
\forall k \in \{1, \dots, N\}, t_k = 2 t_{k-1}
\]
and $t_N = T = \frac{s+1}{2s+d}\log(n)$. Remark that we must have $N \leq C \log(n)$. Let us introduce $(\tilde{X}_t)_{t \in [0,T]}$ defined by $\tilde{X}_0 \sim \gamma$ and 
\[
d  \tilde{X}_t = \big( \tilde{X}_t + 2 \score^\star_{T-t}( \tilde{X}_t)\big) dt + \sqrt{2} dW_t.
\]
For $t \in [0,T]$, we denote by $\tilde{\mu}_t$ the law of $\tilde{X}_{T -t}$ and by $(\hat{\mu}^{SGM}_n)_t$ the law of $\widehat{\overleftarrow{X}}_{T-t}$, defined in Equation~(\ref{def:diffusion-munSGM}). We can then decompose the error into four terms:
\begin{multline*}
\E[W_1(\hat{\mu}^{SGM}_n, \mu)] \leq \E[W_1(\hat{\mu}^{SGM}_n,(\hat{\mu}^{SGM}_n)_{t_0})] + \E[W_1((\hat{\mu}^{SGM}_n)_{t_0}, \tilde{\mu}_{t_0})] \\
+ W_1(\tilde{\mu}_{t_0}, \mu_{t_0})  +  W_1(\mu_{t_0}, \mu).
\end{multline*}
The first and fourth terms correspond to an early stopping error and are easy to deal with. For the fourth term, by Equation~(\ref{eq:OUconvo}), we directly have 
\[
W_1(\mu_{t_0}, \mu) \leq C \sqrt{t_0} \leq Cn^{-\frac{s+1}{2s+d}}.
\]
On the other hand, for the first term, remarking that $\|\tilde{\score}_t\|_{L^\infty(\R^D)} \leq \sqrt{\frac{4 \log(n)}{1-e^{-2t}}}$, using It\^{o}'s lemma yields
\[
W_1(\hat{\mu}^{SGM}_n,(\hat{\mu}^{SGM}_n)_{t_0}) \leq C \sqrt{t_0 \log(n)} \leq C \sqrt{\log(n)} n^{-\frac{s+1}{2s+d}}.
\]
To deal with the third term, since $T = \frac{s+1}{2s+d}\log(n)$, we can use Lemma D.6 \cite{oko} to obtain
\[
W_1(\tilde{\mu}_{t_0}, \mu_{t_0}) \leq C e^{-T} \leq C n^{-\frac{s+1}{2s+d}}.
\]
We are thus left with dealing with the second term. Applying Lemma D.7 \cite{oko} yields
\begin{multline*}
\E[W_1((\hat{\mu}^{SGM}_n)_{t_0}, \tilde{\mu}_{t_0})] \leq \\
C \left(n^{-\frac{s+1}{2s+d}} \sqrt{\log(n)} + \sum_{k=1}^{N} \sqrt{t_k \log(n) \int_{t_{k-1}}^{t_k} \E[\|\tilde{\score}_t - \score^\star_t\|_{L^2(\mu_t)}^2] \, dt} \right).
\end{multline*}
Since $N$ is of order $\log(n)$, we are thus left with showing
\begin{equation}
\label{eq:objective}
\forall k \in \{1, \dots, N\}, t_k \int_{t_{k-1}}^{t_k} \E[\|\tilde{\score}_t - \score^\star_t\|_{L^2(\mu_t)}^2] \, dt \leq C n^{-\frac{2(s+1)}{2s+d}}.
\end{equation}

Let us first take $k \in \{1, \dots, N\}$ such that $t_{k} \leq \frac{C'}{\log(n)}$. Let us start by noting that 
\[
\|\tilde{\score}_t - \score^\star_t\|_{L^2(\mu_t)} = \|c(\hat{\score}_t) - \score^\star_t\|_{L^2(\mu_t)} \leq  \|c(\hat{\score}_t) - c(\score^\star_t)\|_{L^2(\mu_t)} + \|c(\score^\star_t) - \score^\star_t\|_{L^2(\mu_t)}.
\]
Then, since $c$ is an orthogonal projection on the ball with radius $\sqrt{\frac{4\log(n)}{1-e^{-2t}}}$ which is convex,  
\[
\|c(\hat{\score}_t) - c(\score^\star_t)\| \leq \|\hat{\score_t} - \score^\star_t\|.
\]
On the other hand, by Equation~\eqref{eq:conditionalscoredef} along with Jensen's inequality,
\begin{align*}
\|c(\score^\star_t) - \score^\star_t\|_{L^2(\mu_t)}^2 & \leq 2 \|\score^\star_t 1_{\|\score^\star_t\|^2 \geq \frac{4\log(n)}{1-e^{-2t}}}\|_{L^2(\mu_t)}^2\\
 & \leq \frac{2}{1-e^{-2t}} \E\left[\|\E[Z_t \mid \overrightarrow{X}_t]\|^2 1_{\|\E[Z_t \mid \overrightarrow{X}_t]\|^2 \geq 4\log(n)}\right] \\
 & \leq  \frac{2}{(1-e^{-2t}) n} \E\left[\|\E[Z_t \mid \overrightarrow{X}_t]\|^2 e^{\|\E[Z_t \mid \overrightarrow{X}_t]\|^2/4} \right]  \\
& \leq  \frac{2}{(1-e^{-2t}) n} \E\left[\|Z_t \|^2 e^{\|Z_t\|^2/4} \right]  \\
 & \leq \frac{2 C}{(1-e^{-2t})n}.
\end{align*}
Therefore, 
\[
t_k \int_{t_{k-1}}^{t_k} \E[\|\tilde{\score}_t - \score^\star_t\|_{L^2(\mu_t)}^2] \, dt \leq 2 t_k \int_{t_{k-1}}^{t_k} \E[\|\hat{\score}_t - \score^\star_t\|_{L^2(\mu_t)}^2] \, dt + C n^{-\frac{2(s+1)}{2s+d}}.
\]
Now, remember that $\window = n^{-\frac{1}{2s+d}}$ and let $m \in \N$ such that $2(s+2+m) > d$. Thus, if $t_{k-1} \leq \frac{C}{\log(n)}$, then Theorem~\ref{thm:maindiffusion} gives 
\[
\int_{t_{k-1}}^{t_k} \E[\|\tilde{\score}_t - \score^\star_t\|_{L^2(\mu_t)}^2] \, dt \leq \int_{t_{k-1}}^{t_k} \frac{C}{(1-e^{-2t})^2} \left( n^{-\frac{2(s+1)}{2s+d}} + r^2 \right) \, dt,
\]
where
\[
r^2 = \begin{cases}
\left(n^{2(s+2+\ell)} (1-e^{-2t})^{d-2(s+2+\ell)} \window^{2\ell d}\right)^{-\frac{1}{2(s+2+\ell)+d}} \text{ if $1-e^{-2t} \leq n^{-\frac{2}{2s+d}}$} \\
   \left(n^{2(s+2+m)} (1-e^{-2t})^{d(m+1)-2(s+2+m)} \right)^{-\frac{1}{2(s+2+m)+d}}  \text{ otherwise}
\end{cases},
\]
which is smaller than $1-e^{-2t}$ since $t \geq t_0$ and $2(s+\ell+2) > d$. Then, by the definition of the $(t_k)_{k \in \{0, \dots, N\}}$ and since $1-e^{-2t} \sim_{t \rightarrow 0} 2t$, we have 
\[
t_k \int_{t_{k-1}}^{t_k} \frac{Cn^{-\frac{2(s+1)}{2s+d}}}{(1-e^{-2t})^2}  \, dt \leq C n^{-\frac{2(s+1)}{2s+d}}.
\]
On the one hand, if $t_k \leq n^{-\frac{2}{2s +d}}$, we have 
\[
t_k \int_{t_{k-1}}^{t_k} \frac{r^2}{(1-e^{-2t})^2} \, dt \leq C \left(n^{2(s+2+\ell)} t_k^{d-2(s+2+\ell)} \window^{2 \ell d}\right)^{-\frac{1}{2(s+2+\ell)+d}}.
\]
Then, since $\window = n^{-\frac{1}{2s + d}}$, $2(s+2+\ell) > d$ and $t_k \leq n^{-\frac{2}{2s+d}}$,
\begin{align*}
\left(n^{2(s+2+\ell)} t_k^{d - 2(s+2+\ell)} \window^{2 \ell d}\right)^{-\frac{1}{2(s+2+\ell)+d}} & \leq \left(n^{2(s+2+\ell)} n^{\frac{4(s+2+\ell) - 2d}{2s+d}} n^{-\frac{2 \ell d}{2s+d}}\right)^{-\frac{1}{2(s+2+\ell)+d}} \\
& \leq n^{-\frac{2(s+2+\ell)(2(s+1)+d) - 2d (1+\ell)}{(2s+d)(2(s+2+\ell) + d)}} \\
& \leq n^{-\frac{2(s+1)(2(s+1)+d) + 4(1+\ell)(s+1)}{(2s+d)(2(s+2+\ell) + d)}} \\
& \leq n^{-\frac{2(s+1)}{2s+d}}.
\end{align*}
Therefore, when $t_k \leq n^{-\frac{2}{2s+d}}$, we have shown that 
\[
t_k \int_{t_{k-1}}^{t_k} \frac{r^2}{(1-e^{-2t})^2} \, dt \leq C n^{-\frac{2(s+1)}{2s+d}}.
\]
Let us now suppose $n^{-\frac{2}{2s+d}} < t_k \leq \frac{C'}{\log(n)}$. Since $2(s+2+m) > d$ we have 
\[
r^2 =  \left(n^{2(s+2+m)} (1-e^{-2t})^{d(1+m) - 2(s+2+m)} \right)^{-\frac{1}{2(s+2+m)+d}}
\]
and thus 
\[
t_k \int_{t_{k-1}}^{t_k} \frac{r^2}{(1-e^{-2t})^2} \, dt \leq C \left(n^{2(s+2+m)} t_k^{d(1+m) - 2(s+2+m)} \right)^{-\frac{1}{2(s+2+m)+d}} .
\]
Now, since $t_k \geq n^{-\frac{2}{2s+d}}$ and, for $m$ large enough such that $2(s+2+m) < d(1+m)$ (such an $m$ exists since $d \geq 3$), 
\begin{multline*}
 \left(n^{2(s+2+m)} t_k^{d(1+m) - 2(s+2+m)} \right)^{-\frac{1}{2(s+2+m)+d}} \leq \\
 \left(n^{s+2+m} n^{-\frac{d(1+m) - 2(s+2+m)}{2s+d}} \right)^{-\frac{2}{2(s+2+m)+d}} = n^{-\frac{2(s+1)}{2s+d}}
\end{multline*}
Thus, we have shown that, as long as $t_k \leq \frac{C'}{\log(n)}$,
\[
t_k \int_{t_{k-1}}^{t_k} \E[\|\tilde{\score}_t - \score^\star_t\|_{L^2(\mu_t)}^2] \, dt \leq  C n^{-\frac{2(s+1)}{2s+d}}.
\]
Let us finally suppose $t_k > \frac{C'}{\log(n)}$. Take $m \in \mathbb{N}$ sufficiently large. Then, by Theorem~\ref{thm:maindiffusion} and Remark~\ref{rem:large-times}, we have 
\[
\int_{t_{k-1}}^{t_k} \E\left[\|\tilde{\score}_t - \score^\star_t \|_{L^2(\mu_t)}^2\right] \, dt \leq \int_{t_{k-1}}^{t_k}  \frac{C}{(1-e^{-2t})^2} \left(r^2  + \frac{1}{n} \right) \, dt,
\]
where 
\[
r^2 =  \left(n^{2m} (1-e^{-2t})^{m(d-2)} \right)^{-\frac{1}{2m+d}}.
\]
And through similar computations as before, since $\frac{C'}{\log(n)} \leq t_k \leq C \log(n)$ and $1-e^{-2t} \leq 1$, we have
\[
t_k \int_{t_{k-1}}^{t_k} \frac{r^2}{(1-e^{-2t})^2} \, dt \leq C \log(n)^2 \left(n^{2m} \log(n)^{(2-d)m} \right)^{-\frac{1}{2m+d}}
\]
and thus, 
\[
t_k \int_{t_{k-1}}^{t_k} \frac{r^2}{(1-e^{-2t})^2} \, dt \leq C n^{-\frac{2(s+1)}{2s+ d}}.
\]
We have shown that Equation~\eqref{eq:objective} holds, which concludes the proof.

\section{Acknowledgements}
 A preliminary exploration of these questions was done during the internship of Lê Nhat Sinh, under the supervision of the authors.

T.B. and V.C.T. are supported by Labex B\'ezout (ANR-10-LABX-58) and by GdR GeoSto 3477. V.C.T. is also supported by the R-CDP-24-004-C2EMPI project, funded by the French State under the France-2030 programme, the University of Lille, the Initiative of Excellence of the University of Lille and the European Metropolis of Lille.

\bibliographystyle{plain}
\bibliography{biblio.bib}

\end{document}